\documentclass[10pt]{amsart}%,reqno
\usepackage[left=3cm,top=2cm,right=3cm,bottom=2cm]{geometry}
\usepackage{amssymb}
\usepackage{amsbsy}
\usepackage{epsfig}
\usepackage{tikz}
\usepackage[english]{babel}    % para texto em inglês
\usepackage{enumitem}
\usepackage{multicol}
\usepackage{blkarray}
\usepackage{tikz}
\usepackage{mathtools}
\usepackage{amsthm}
\usepackage{float}
\usepackage{algorithm}
\usepackage{algpseudocode}

\newtheorem{teore}{Theorem}[section]
\newcommand{\sign}{\mbox{sign}}

\newtheorem{definition}[teore]{Definition}
\newtheorem{lemma}[teore]{Lemma}
\newtheorem{corollary}[teore]{Corollary}
\newtheorem{theorem}[teore]{Theorem}

\newtheorem{remark}[teore]{Remark}
\newtheorem{example}[teore]{Example}

\begin{document}
\pagestyle{myheadings}
%********************************************************

\title[Limit points of Laplacian spectral radii]{Limit points of (signless)  Laplacian spectral radii of linear trees}
\subjclass{05C50, 05C05, 15A18}
\keywords{spectral radius; tree;  limit points;}
%%%%%%%%%%%%%%%%%%%%%%%%%%%%%%%%%%%%%%%%%%%%%%%%%%%%%%% nome
\author[F. Belardo]{Francesco Belardo} \address{Department of Mathematics and Applications, University of Naples Federico II, Italy} \email{\tt fbelardo@unina.it}

\author[E. R. Oliveira]{Elismar R. Oliveira}
\address{Instituto de Matem\'atica e Estat\'{\i}stica, UFRGS, Porto Alegre, Brazil}
\email{\tt elismar.oliveira@ufrgs.br}

\author[V. Trevisan]{Vilmar Trevisan}
\address{Instituto de Matem\'atica e Estat\'{\i}stica, UFRGS, Porto Alegre, Brazil }
\email{\tt trevisan@mat.ufrgs.br}

\begin{abstract}
We study limit points of the spectral radii of Laplacian matrices of graphs. We adapted the method used by J. B. Shearer in 1989, devised to prove the density of adjacency limit points of caterpillars, to Laplacian limit points. We show that this fails, in the sense that there is an interval for which the method produces no limit points. Then we generalize the method to Laplacian limit points of linear trees and prove that it generates a larger set of limit points. The results of this manuscript may provide important tools for proving the density of Laplacian limit points in $[4.38+, \infty)$.
\end{abstract}

\maketitle
%%%%%%%%%%%%%%%%%%%%%%%%%%%%%%%%%%%%%%%%%%%%%%%%%%%%%%%%%%%%%%%%%%%%%%%%%%%%%%%%%%%%%%%%%%%%%%%%%%%%%%%%%%%%%%%%%%%%%%%%%%%%%%%%%%%%%%%%%%%%%

\section{Introduction}

The 1972 seminal paper of A. J. Hoffman~\cite{hoffman1972limit} introduced the concept of limit points of eigenvalues of graphs. Let $\mathcal{A}$ be the set of all symmetric matrices of all orders, every entry of which is a non-negative integer and $R=\{\rho:\rho=\rho(A) \mbox{ for some } A\in\mathcal{A}\}$ where $\rho(A)$ is the largest eigenvalue of $A$. He asked which real number can be in $R$ and showed that it is sufficient to consider matrices of $\mathcal{A}$ having only entries in $\{0, 1\}$ and 0 diagonal, e.g. adjacency matrices of graphs. Additionally, he determined all limit points of $R \leq\sqrt{2+\sqrt{5}}$. More precisely, let $\tau=\frac{1+\sqrt{5}}{2}$ (the golden mean). For $n=1,2, \ldots$, let $\overline{\beta}_n$ be the positive root of
$$
Q_n(x)=x^{n+1}-\left(1+x+x^2+\cdots+x^{n-1}\right).
$$
Let $\bar{\alpha}_n=\bar{\beta}_n^{1 / 2}+\bar{\beta}_n^{-1 / 2} \cdot$ Then
$$
2=\bar{\alpha}_1<\bar{\alpha}_2<\cdots
$$
are all the limit points of $R$ smaller than ${\displaystyle\lim _{n \to \infty} \bar{\alpha}_{n}=\tau^{1 / 2}+\tau^{-1 / 2}=\sqrt{2+\sqrt{5}}  ( =2.05+)}$.\\

In 1989,  a remarkable result due to J. B. Shearer~\cite{shearer1989distribution} extended the work of Hoffman. He showed that every real number larger than $\sqrt{2+\sqrt{5}}$ is a limit point of $R$.

In his original paper, Hoffman asks about limit points of other eigenvalues of graphs, and this originated several interesting results on this topic culminating with the fact that every real number is a limit point of some eigenvalue of a graph (see  Zhang \& Chen~\cite{zhang2006limit}). We remark that Estes~\cite{ESTES1992} proved in 1992 that every totally real algebraic integer -- which are roots of monic integral polynomials having only real roots -- is an eigenvalue of a graph.  Since eigenvalues are totally real algebraic integers (roots of the characteristic polynomial), it follows that this set is dense in the real line. What is perhaps surprising, is that only eigenvalues of trees are needed to be considered for these results. Salez~\cite{Salez2015}, extending Estes' result, showed that any totally real algebraic integer is an eigenvalue of a tree, while both Hoffman and Shearer found sequences of trees whose spectral radii were limit points. It seems remarkable that the set composed only by the largest eigenvalue of the adjacency matrix of trees is dense in the interval $[\sqrt{2+\sqrt{5}}, \infty)$.

It is worth noticing that in 2003, Kirkland \cite{kirkland2003note} showed that any nonnegative real number is a limit point for the algebraic connectivity (the second smallest Laplacian eigenvalue). The same author \cite{kirkland2006limit} considered limit points for the positive eigenvalues of the normalized Laplacian matrix of a graph. Specifically, he shows that the set of limit points for the $j$-th smallest such eigenvalues is equal to $[0,\, 1]$, while the set of limit points for the $j$-th largest such eigenvalues is equal to $[1,\, 2]$.  Wang et al. \cite{Bel2010} proved that any nonnegative real number is a limit point for some eigenvalue of the signless Laplacian matrix of graphs.

In this paper, we are interested in the problem originated by Hoffman's question, which deals only with limit points of the spectral radius of graphs. More specifically, we want to study the Laplacian version of Hoffman and Shearer's results, that is, what real numbers are limit points of the spectral radii of Laplacian matrices of graphs. The converse of this problem may also be viewed as \emph{which sequences of graphs have Laplacian spectral radius with limit points}  (we refer to the next section for the precise definitions).

In order to explain our results, we recall the work of Guo~\cite{guo2008limit}. Let $$\omega=\frac{1}{3}(\sqrt[3]{19+3 \sqrt{33}}+\sqrt[3]{19-3 \sqrt{33}}+1),$$ $~ \beta_0=1$ and $\beta_n,\;n \geqslant 1$ be the largest positive root of
$$
P_n(x)=x^{n+1}-\left(1+x+\cdots+x^{n-1}\right)(\sqrt{x}+1)^2 .
$$
Let $\alpha_n=2+\beta_n^{\frac{1}{2}}+\beta_n^{-\frac{1}{2}}$. Then
$$
4=\alpha_0<\alpha_1<\alpha_2<\cdots
$$
are all the limit points of Laplacian spectral radii of graphs smaller than $\lim _{n \rightarrow \infty} \alpha_n=2+\omega+\omega^{-1}$ (= $4.38+)$.

By analogy to the adjacency case, it is natural to ask whether any real number $\mu \geq 2+w+w^{-1}=4.38+$ is the limit point of the Laplacian spectral radii of graphs. We refer to Figure~\ref{fig:esquema} for an illustration of the current state of knowledge.
\begin{figure}[ht!]
  \centering
  \includegraphics[width=12cm]{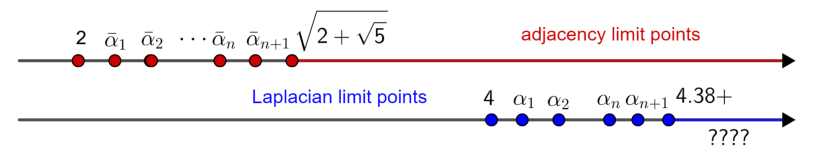}
  \caption{Current status of (Laplacian) limit points}\label{fig:esquema}
\end{figure}

Our contribution in this paper is the development of some analytical tools allowing one to study the density of Laplacian spectral radius in $[4.38+, \infty)$. We first adapt Shearer's method to the Laplacian case, verifying that it is not sufficient to prove density. In spite of the fact that the method produces sequence of caterpillars having a limit point, we show that this limit point is not always the \emph{desired} number. As a consequence, we find a whole interval where the method produces no Laplacian limit points. We then extend Shearer's method to the class of linear trees, e.g. we define sequences of linear trees whose Laplacian spectral radius has limit points. This provides a generalization of Shearer's process, since caterpillars is a subclass of linear trees.

It is tempting to conjecture that any $\mu \geq 4.38+$ is a limit point of Laplacian spectral radii of graphs, however, it seems that Laplacian spectral radii are a bit slick! Our generalization improves Shearer's process in the sense that we are able to find a larger set of limit points. We do not know whether sequences of linear trees suffice to prove the density conjecture (if true).

The rest of the paper is organized as follows. In the next section, we present the necessary notation, definitions and the main tools used in the paper. In particular, we explain Shearer's procedure using our tools. In Section~\ref{sec:revisit shearer Laplacian}, we apply Shearer's approach to Laplacian case, verifying that it fails. In Section~\ref{sec:lintreeseq} we study the convergence of the Laplacian spectral radius of sequences of linear trees. In Section~\ref{sec:appProcess} we use these results to generalize Shearer's work: we obtain a method that produces a sequence of linear trees for a given number $\mu$ in a way that the Laplacian spectral radius has a limit point $\leq \mu$. We call this \emph{generalized Shearer's sequence} and show that the process provide a larger set of limit points.  In Section~\ref{sec:density laplac lim points} we study the Laplacian limit points given by the method. We present some numerical evidence of the density of these limit points. We also prove a density result, showing that a truncated sequence of linear trees have algebraic limit points that is dense in a subset of the limit points. In Section~\ref{sec:vartec}, we apply some analytical tools to obtain approximations, providing quicker ways to obtain sequence of linear trees having a specified limit point. Finally, in Section~\ref{sec:final}, we suggest a few open problems and make final considerations.

\begin{remark}
 The results of this paper are stated for the Laplacian matrix, however, as all the graphs involved are trees, they naturally extend to signless Laplacian matrix, since the spectrum of these two matrices are equal for bipartite graphs.
\end{remark}

\section{Notation and preliminaries} \label{sec:prel}
Let $G=(V,E)$ be an undirected graph with vertex set $V$ and edge set $E$. If $|V|=n$, then its adjacency matrix $A(G)=[a_{ij}]$ is the $n\times n$ matrix of zeros and ones such that $a_{ij}=1$ if and only if $v_{i}$ is adjacent to $v_{j}$ (that is, there is an edge between $v_{i}$ and $v_{j}$). A value $\lambda$ is an eigenvalue if ${\rm det}(\lambda I_{n}-A)=0$, and since $A$ is real symmetric its eigenvalues are real. In this setting, we denote by $\rho_{A}(G)$ the largest eigenvalue of $A(G)$, which is called the spectral radius of $G$. One can also consider the Laplacian matrix $L(G)$ of a graph, which is given by $L(G):= D(G)-A(G)$, where  $D(G)=[d_{ij}]$ is the $n\times n$ diagonal matrix with $d_{ii}={\rm degree}(v_{i})$. A value $\mu$ is an eigenvalue if ${\rm det}(\mu I_{n}-L)=0$, and since $L$ is positive semi-definite, its eigenvalues are non-negative. We denote by $\rho_{L}(G)$ the largest eigenvalue of $L(G)$ which is called the Laplacian spectral radius of $G$. One can also consider the signless Laplacian associated to a graph, which is denoted $Q(G):= D(G)+ A(G)$.\\

Generalizing the concept of limit point of graphs, we say that a real number $\gamma$ is an $M$-limit point of the $M$-spectral radius of graphs if there exists a sequence of graphs $\{G_k \; | \; k \in \mathbb{N}\}$ such that
$$\lim_{k \to \infty} \rho_{M}(G_k)=\gamma,$$
where $\rho_{M}(G_i)\neq \rho_{M}(G_j),\; i\neq j$ and $M$ is a class of matrices associated with a graph such as adjacency, Laplacian, signless Laplacian, etc. (See  Wang \& Brunetti \cite{wang2020hoffman}).

In this note we will extend Shearer's ideas for the adjacency matrix to the Laplacian version. The background of this process can be recovered by the use of modern methods such as the Jacobs-Trevisan diagonalization algorithm~\cite{jacobs2011locating}, and properties of some recurrence equations from \cite{OliveTrevAppDiff}. Those are our main tools and we briefly explain here, so that the note is self contained.

Given a symmetric matrix $A$ whose underlying graph is a tree, the following algorithm of Figure~\ref{fig:algo} outputs a diagonal matrix that is congruent to $A+xI$ and hence, the sign of its diagonal values determine the number of eigenvalues that greater than/equal to/larger than $x$. More precisely, we state the following result for future reference.

\begin{lemma}(\cite{jacobs2011locating}) \label{lem:loc} Let $A$ be a symmetric matrix whose underlying graph is a tree and let $D$ be the matrix output by $\rm{Diagonalize}(A,-x)$.
The number of eigenvalues of $A$  smaller than/equal to/larger than $x$ is, respectively,  the number of negative/zero/positive diagonal values of $D$
\end{lemma}

\begin{figure}[h]
{\tt
\begin{tabbing}
aaa\=aaa\=aaa\=aaa\=aaa\=aaa\=aaa\=aaa\= \kill
     \> Input: matrix $A = (a_{ij})$ and underlying tree $T$ whose vertices $v_1,\ldots,v_n$ \\
     \> are ordered bottom-up and real $x$\\
     \> Output: diagonal matrix $D = \mbox{diag}(d_1, \ldots, d_n)$ congruent to $A+xI$ \\
     \> \\
     \>  Algorithm $\mbox{Diagonalize}(A,x)$ \\
     \> \> initialize $d_i \longleftarrow a_{ii}+x$, for all $i$ \\
     \> \> {\bf for } $k = 1$ to $n$ \\
     \> \> \> {\bf if} $v_k$ is a leaf {\bf then} continue \\
     \> \> \> {\bf else if} $d_c \neq 0$ for all children $c$ of $v_k$ {\bf then} \\
     \> \> \>  \>   $d_k \longleftarrow d_k - \sum \frac{(a_{ck})^2}{d_c}$, summing over all children of $v_k$ \\
     \> \> \> {\bf else } \\
     \> \> \> \> select one child $v_j$ of $v_k$ for which $d_j = 0$  \\
     \> \> \> \> $d_k  \longleftarrow -\frac{(a_{jk})^2}{2}$ \\
     \> \> \> \> $d_j  \longleftarrow  2$ \\
     \> \> \> \> if $v_k$ has a parent $v_\ell$, delete the edge $\{v_k,v_\ell\}$. \\
     \> \>  {\bf end loop} \\
\end{tabbing}
}
\caption{\label{fig:algo} Diagonalizing $A$ for a symmetric
matrix $A$ with an underlying tree.}
\end{figure}

Applying the tree algorithm on a path of a tree, produces certain numerical rational sequences. They all may be seen in a unified elementary form given by
\begin{equation}\label{formphi}
  x_{j+1}= \varphi(x_{j}), j \geq 1
\end{equation}
where $\varphi(t)= \alpha + \frac{\gamma}{t}$, for $t\neq 0$,   $\alpha, \gamma \in \mathbb{R}$ are fixed numbers ($\gamma \neq 0$) and $x_{1}$ is a given initial condition. The explicit solution is a function $j \to f(j)$ such $x_{j}= f(j)$ for $j \geq 1$ and has been studied in \cite{OliveTrevAppDiff}.

\subsection{Revisiting Shearer's work  for adjacency limit points}
\label{sec:revisit shearer}

In \cite{shearer1989distribution} Shearer has proved the density of the limit points for adjacency matrices in the interval $[\sqrt{2+\sqrt{5}}, \infty)$. The proof is based on an explicit construction of a sequence of graphs. For any given $\lambda \geq \sqrt{2+\sqrt{5}}$, Shearer constructs a sequence of caterpillars whose spectral radii converge to $\lambda$. We now explain, in our framework, why this method works.

Recall that a \emph{caterpillar} is a tree in which the removal of all its leaves transforms it into a path. An arbitrary caterpillar with $k$ back nodes $v_1,\ldots, v_k$, where $v_i$ has $r_i$ leaves may be represented by $$G_{k}:=[r_{1}, r_{2}, ..., r_{k}], \; k \geq 2.$$ We refer to  Figure~\ref{fig:caterpillar} for an illustration.
\begin{figure}
  \centering
  \includegraphics[width=10cm]{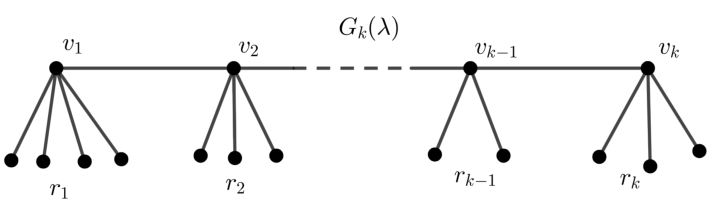}
  \caption{Representation of $G_{k}$}\label{fig:caterpillar}
\end{figure}

We notice that, if we replace each $r_i$ by a star tree with $r_i$ rays, we have an alternative representation for $G$. For instance
$$[3, 1, 0, 2, 2] = [[1, 1, 1], [1], [0], [1, 1], [1, 1]]$$ are different representations of the same tree.

For each $\lambda \in [\sqrt{2+\sqrt{5}},  \infty)$ Shearer finds a caterpillar $G_k(\lambda)=[r_{1}, r_{2}, ..., r_{k}], \; k \geq 2$, in such a way that the number of leaves in each vertex satisfies a recurrence related to the number $\frac{-\lambda\pm\sqrt{\lambda^2 -4}}{2}$.

Consider a given number $\lambda \in [\sqrt{2+\sqrt{5}},  \infty)$ and an arbitrary caterpillar $G_{k}(\lambda):=[r_{1}, r_{2}, ..., r_{k}], \; k \geq 2.$ We apply the algorithm ${\rm Diagonalize}(A(G_{k}), -\lambda)$, the adjacency matrix of $G_k$. Each leaf is initialized with $-\lambda <0$. In the vertex  $v_{1}$ we obtain the value
$$R_{1}:= -\lambda - \frac{r_{1}}{-\lambda}=-\lambda + \frac{r_{1}}{\lambda},$$
and in the next vertex $v_{2}$ we obtain the value
$$R_{2}:= -\lambda - \frac{r_{2}}{-\lambda} - \frac{1}{R_{1}}= -\lambda - \frac{1}{R_{1}} + \frac{r_{2}}{\lambda},$$
and so on.

We observe that the sequence of equations may be seen as the recurrence relation
\begin{equation}\label{eq:base recurrence adjacency}
  \left\{
  \begin{array}{ll}
    x_{j+1}=\varphi(x_{j})\\
    x_{1}=-\lambda,
  \end{array}
  \right.
\end{equation}
where $\varphi(t)=-\lambda -\frac{1}{t}, \; t\neq 0$, disturbed by a drift factor $\delta_{j}:=\frac{r_{j}}{\lambda} \geq 0$.

From \cite{OliveTrevAppDiff} we know that the iterates $x_j$ obeying  Equation~\eqref{eq:base recurrence adjacency} are asymptotic to $\theta=\frac{-\lambda -\sqrt{\lambda^2 - 4}}{2}$, which is a fixed point $\varphi(t)=t$, satisfies  $t^2 + \lambda t +1 =0$. It is shown that
$\theta:= \frac{-\lambda -\sqrt{\lambda^2 - 4}}{2}$ is an attracting point and $\theta':= \frac{-\lambda +\sqrt{\lambda^2 - 4}}{2}= \theta^{-1}>\theta$ is a repelling point, with respect to the iteration process \eqref{eq:base recurrence adjacency}.

From Lemma~\ref{lem:loc}, it follows that $\rho_{A}(G_{k}) < \lambda$ if and only if $R_j <0$ for all $j$. As the attracting interval for $\theta$ is $(-\infty, \theta']$ we need to require that $R_j \leq \theta'$ otherwise  $R_{j'} >0 $ for some $j'>j$ contradicting $\rho_{A}(G_{k}) < \lambda$.

This is the basis for the Shearer's construction (see \cite{shearer1989distribution} equations (1)-(4)), that is, the sequential choice of $r_1, r_2,...,r_k$:\\
If $ -\lambda + \frac{r_{1}}{\lambda} \leq \theta'$ then $R_{1}<0$,
thus we choose
$$r_1:= \max_{r \geq 0}  \{-\lambda + \frac{r}{\lambda} \leq \theta' \} \Leftrightarrow r_1:= \lfloor \lambda ( \theta' - (-\lambda))\rfloor.$$
If $ \varphi(R_{1}) + \frac{r_{2}}{\lambda} \leq \theta',$ then  $ R_{2}<0$,
thus we choose
$$r_2:= \max_{r \geq 0}  \{\varphi(R_{1}) + \frac{r_{2}}{\lambda} \leq \theta' \}\Leftrightarrow r_2:= \lfloor \lambda ( \theta' - \varphi(R_{1}))\rfloor,$$
and so on.
\begin{definition}\label{def: adjacency shaerer sequence}
   Given $\lambda \geq \sqrt{2+\sqrt{5}}$ we define the adjacency Sharer's sequence as the sequence of caterpillars $G_{k}(\lambda):=[r_{1}, r_{2}, ..., r_{k}], \; k \geq 2$ where \[
   \left\{
     \begin{array}{ll}
       r_1:= \lfloor \lambda ( \theta' - (-\lambda))\rfloor; \\
       R_{1}=-\lambda + \frac{r_{1}}{\lambda};\\
       r_{j}:= \lfloor \lambda ( \theta' - \varphi(R_{j-1}))\rfloor, 2 \leq j \leq k;\\
       R_{j}=\varphi(R_{j-1}) + \frac{r_{j}}{\lambda}, 2 \leq j \leq k.
     \end{array}
   \right.
\]
\end{definition}
The adjacency Sharer's sequence $G_{k}, \; k \geq 2$ is such that $\rho_{A}(G_{k}) < \rho_{A}(G_{k+1}) < \lambda$.  Indeed, the part $\rho_{A}(G_{k}) < \lambda$ is a trivial consequence of the construction because $R_1,..., R_k <\theta'<0$ then the spectral radius is necessarily smaller than $\lambda$ (see Theorem 3, from \cite{jacobs2011locating}). The part $\rho_{A}(G_{k}) < \rho_{A}(G_{k+1})$ is a consequence of the interlacing property because $G_{k}$ is always a subgraph of $G_{k+1}$ since the choice of $r_k:= \lfloor \lambda ( \theta' - \varphi(R_{k-1}))\rfloor$ is the same for both graphs.

In particular there exists
$$\lim_{k \to \infty} \rho_{A}(G_k)=\gamma(\lambda) \leq \lambda.$$
Shearer~\cite{shearer1989distribution}  showed that for every real number $\lambda \geq\sqrt{2+\sqrt{5}}$, the limit point $\gamma(\lambda) =\lambda$, that is $\displaystyle \lim_{k\to \infty} \rho_{A}(G_k)=\lambda$.

In this note we will adapt Shearer's construction to the Laplacian matrix. For this, we set next some notation that is going to be used throughout the paper.

\subsection{Notation for Laplacian Matrices}\label{sec:gennotation}

From Guo's result~\cite{guo2008limit}, the problem for limit points of spectral radii of Laplacian matrices that remains to study is the interval $\mu \geq 4.38+$. More precisely, the question is whether any real number $\mu \in [4.38+,\infty)$ is a limit point of Laplacian spectral radius of graphs.

In the next section, we extend Shearer's method of caterpillar construction for Laplacian matrices. We will see the necessity to consider other trees. Here, we set the notation for this larger set of trees.

We recall that a \emph{linear tree} is a tree obtained by attaching a starlike tree (that has at most one vertex of degree larger than 2) at each vertex of a path.  We denote a linear tree by $G=[T_{1},\ldots, T_{k}]$, for $k \geq 1$, where all high degree vertices (HDVs) are located in one single path $P=\{v_{1} \to v_{2} \to  \cdots \to v_{k}\}$, called the main path. At each vertex $v_{j}$ (we call then back nodes), we attach a number of paths $P_{q_{1}^{j}},\ldots ,P_{q_{r_{j}}^{j}}$ forming a starlike tree, which is denoted by $T_{j}=[q_{1}^{j}, ...,q_{r_{j}}^{j}], q_{i}^{j}\leq q_{i+1}^{j}$, for $1 \leq j \leq k$ (see Figure~\ref{fig:linear_tree_example} for an illustration).

The class of linear trees was introduced by Johnson, Li  \& Walker in \cite{JOHNSON2014} and further studied by Johnson \& Saiago (see \cite{johnson2018eigenvalues} Chapter 10). Some special cases of linear trees are  caterpillars, where all the HDVs have a number of pendant leafs $P_1$, and the open quipus, where all the HDVs have degree at most three.

\begin{figure}[ht!]
  \centering
  \includegraphics[width=13cm]{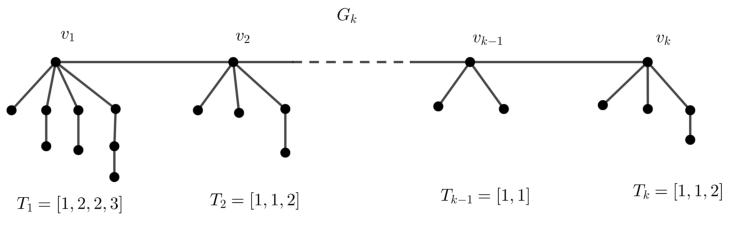}
  \caption{Linear tree $G_{k}=[[1,2,2,3],[1,1,2],\ldots,[1,1],[1,1,2]]$}\label{fig:linear_tree_example}
\end{figure}

We notice that the representation $T_{j}=[0]$ means that the graph $T_{j}$ is formed just by the vertex $v_{j}$, or that $T_{j}$ is empty.

We denote by $\ell(G_{k})=k$, the \emph{length} of $G=[T_{1},\ldots, T_{k}]$, as the number of back nodes in the main path,  by $\omega(T_{j})=r_{j}$, the \emph{width} of $T_{j}=[q_{1}^{j}, \ldots ,q_{r_{j}}^{j}]$, as the number of paths it is composed, and by $h(T_{j})=q_{r_{j}}^{j}$, the height of $T_{j}=[q_{1}^{j}, ...,q_{r_{j}}^{j}]$, as the maximum length of paths in it.

The representation $G_{k}=[T_{1}, ..., T_{k}]$ is almost never unique unless $T_1=T_k=[0]$ and $k$ is fixed, because we may choose another main path by picking paths from $T_1$ and/or $T_k$. For instance $[[2,2,2],[1,2],[1,1]]$ and $[[0],[0],[2,2],[1,2],[1,1]]$ are representations of the same linear tree. Although, in any new representation the length could increase at most by $h(T_{1})+h(T_{k})$.

We consider applying ${\rm Diagonalize}(L,-\mu)$, for $\mu \geq 4.38+$ and $L$ the Laplacian matrix of a linear tree with $k$ back nodes $G_k=[T_1,T_2,\ldots,T_k]$.

When executing ${\rm Diagonalize}(L,-\mu)$, we first apply in $T_i, ~i=1,\ldots, k$, towards the vertex $v_k$, which is the root. We remark that $T_i$ may be empty and in this case there is nothing to be done. If $T_i$ is a non-empty starlike, for each path of $T_i$, we have that leaves receive the value $b_1 = 1-\mu$, while the remaining values satisfy $b_j=2-\mu-\frac{1}{b_{j-1}}$, for $j>1$. We refer to Figure~\ref{fig:Linear} for an illustration.
\begin{figure}[ht!]
  \centering
  \includegraphics[width=13cm]{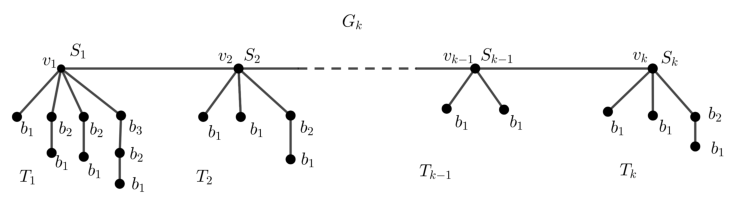}
  \caption{Linear tree with the values obtained by ${\rm Diagonalize}(L,-\mu)$ depicted in each vertex.}\label{fig:Linear}
\end{figure}

It will be convenient to denote, for future reference, the function $\psi: \mathbb{R}\rightarrow \mathbb{R}$ by \begin{equation}\label{eq:psi}
     \psi(t) := 2-\mu -\frac{1}{t}.
   \end{equation}
With this notation, the values in each path of $T_i$ satisfy
\begin{equation}\label{eq:path}
    \left\{ \begin{array}{ll}
               b_1 = 1-\mu\\
               b_{j+1} = \psi(b_j), ~~j>1.
             \end{array}
    \right.
   \end{equation}

From \cite{OliveTrevAppDiff} we know that the iterates $b_j$ obeying Equation~\eqref{eq:path} have the following properties. Let $\theta$ and $\theta'$ be fixed points of $\psi(t)=t$. They are given by
\begin{equation}\label{eq:theta}
\theta:= \frac{-(\mu -2) -\sqrt{(\mu -2)^2 - 4}}{2} \mbox { and } \theta':= \frac{-(\mu -2) +\sqrt{(\mu -2)^2 - 4}}{2}= \theta^{-1}>\theta.
\end{equation}
They are attracting and repelling points, respectively, meaning that for $b_1 \leq  \theta'$, we have $b_j$ converging to $\theta$.
Since $\mu > 4$, we have that $b_1 < \theta'<0$ and, in particular, that all $b_j < 0$. We process each path of $T_i$ until the vertices adjacent to $v_i$. We notice that to process the vertex $v_i$ we need the value resultant from processing $T_i$. We denote by $\omega(T_{i})=r_{i}$ the number of paths of $T_i$ and, for $m=1,\ldots,r_i$, by $b_{i,m}$ the last value output by the algorithm at the $m$-th path of $T_i$. The net result of the starlike $T_i$ is given by
\begin{equation}\label{eq:defdelta}
\delta(T_i,\mu):=\left\{
\begin{array}{ll}
       \sum_{j=1}^{\ell_i} \left(1 -\frac{1}{ b_{i,j}}\right),  \mbox{ for non-empty } T_i \\
       0 \mbox{ if } T_i \mbox{ is empty. }\\
\end{array} \right.
\end{equation}
Continuing applying the algorithm, we go from $v_1$ towards $v_k$, the root of $G_k$. We denote $S_i$, for $i\in\{1,\ldots,k\}$, the value output by  ${\rm Diagonalize}(L,-\mu)$. We have $S_1=1-\mu +\delta(T_1,\mu),~ S_2=2-\mu+\delta(T_2,\mu)-\frac{1}{S_1},~ S_3=2-\mu+\delta(T_2,\mu)-\frac{1}{S_2}$ and so on. For $S_k$ the expression is $S_{k}:= 1-\mu - \frac{1}{S_{k-1}} + \delta(T_k,\mu)$ (see Figure~\ref{fig:Linear} for an illustration). Using the function $\psi$, we notice that $S_i$ obeys the following recurrence.

\begin{equation}\label{eq:Laplacian psi}
  \left\{
  \begin{array}{ll}
    S_{1}=1-\mu + \delta(T_1,\mu),\\
    S_{i}=\psi(S_{i-1})+ \delta(T_2,\mu),\; 2 \leq i \leq k-1,\\
    S_{k}=-1+ \psi(S_{k-1})+ \delta(T_k,\mu).
  \end{array}
  \right.
\end{equation}
We are interested in determining whether $\mu$ is smaller than, equal to, or larger than $\rho_L(G_k)$, the Laplacian spectral radius of $G_k$. Since the values of all $b_j$'s are negative, by Lemma~\ref{lem:loc}, only the signs of the values $S_j$ suffice to determine this. More precisely, consider the function  $$\Pi:=(G_k,\mu) \to \left(S_1, S_2,\ldots,S_k \right),$$
and let $\sign(\Pi)(G_k,\mu)=\left(\sign(S_1), \sign(S_2),\ldots,\sign(S_k) \right)$. From Lemma~\ref{lem:loc} we can state the following result.

\begin{lemma}\label{lem:detmu} With the established notation, we have

\begin{itemize}
   \item[$(a)$] $\rho_L(G_k) < \mu \Leftrightarrow \sign(\Pi)(G_k,\mu)=(-,-,\ldots,-,-)$;
   \item[$(b)$] $\rho_L(G_k) = \mu \Leftrightarrow \sign(\Pi)(G_k,\mu)=(-,-,\ldots,-,0)$;
   \item[$(c)$] $\rho_L(G_k) > \mu \Leftrightarrow \sign(\Pi)(G_k,\mu)$ has a + entry.
\end{itemize}
\end{lemma}

\section{Applying Shearer's approach for Laplacian limit points}
\label{sec:revisit shearer Laplacian}
In the previous section, we set a framework to apply ${\rm Diagonalize}(L(G_{k}), -\mu)$, where $G_k$ is a linear tree and $\mu$ is any real number larger than 4. We are going to use the notation established there, but we observe a subtle difference. Here, from a given number $\mu \geq 4.38+$, we construct a prescribed caterpillar $G_{k}(\mu):=[r_{1}, r_{2}, ..., r_{k}], \; k \geq 2$ (see Figure~\ref{fig:caterpillar}) by adapting Shearer's method to Laplacian matrices. Additionally, we will study the convergence of $\rho_L(G_k(\mu))$.

Let us then fix a real number $\mu \in [4.28+, \infty)$. Shearer's approach needs to determine numbers $r_1, r_2,..., r_k,$ such that $G_k(\mu)=[r_1, r_2,...,r_k]$ has $\rho_L(G_k) < \mu$. In the language of Section~\ref{sec:gennotation}, the linear trees $T_i$ are simply stars, whose paths have length one. As $T_i=[1,1,\ldots,1]$, we have $\delta(T_i,\mu):= \sum_{j=1}^{r_i} \left(1 -\frac{1}{ b_{1}}\right) = \sum_{j=1}^{r_i} \left(1 -\frac{1}{ 1-\mu}\right) = r_{i}\frac{\mu}{\mu -1} = \delta_{i}.$
We recall that, from Lemma~\ref{lem:detmu}, $\rho_{L}(G_{k}) < \mu$ if and only if $S_j <0$ for all $j<k$. As the attracting interval for $\theta$, given by Equation~\eqref{eq:theta}, is $(-\infty, \theta']$ we need to require that $S_j \leq \theta'$ otherwise  $S_{j'} >0 $ for some $j'>j$ contradicting $\rho_{L}(G_{k}) < \mu$.

Following the reasoning of Shearer's construction, we determine the sequential choice of $r_1, r_2,\ldots, r_k $ as follows.\\
Since $S_{1}<0  \text{ if } 1-\mu + r_{1}\frac{\mu}{\mu -1} \leq \theta'$,
we choose
$$r_1:= \max_{r \geq 0} \left\{ 1-\mu + r \frac{\mu}{\mu -1} \leq \theta'\right\} \Leftrightarrow r_1:= \left\lfloor \frac{\mu -1}{\mu} ( \theta' - (1-\mu))\right\rfloor.$$
We point out that here we deal with real inequalities, such as $S_{1}<0$, or equivalently $1-\mu + r_{1}\frac{\mu}{\mu -1} \leq \theta'$. Of course, this inequality is true for $r_{1}=0$. Moreover, if we increase $r$ by one, then the left hand side of the inequality increases by $\frac{\mu}{\mu -1}>0$. Thus, there will be a maximum number $r$ satisfying the inequality. If we assume that $r$ could be a real number, then the equality happens for  the real number $\tilde{r}:= \frac{\mu -1}{\mu} ( \theta' - (1-\mu))$. Clearly, the largest positive integer satisfying this property is $r_1:= \left\lfloor \tilde{r}\right\rfloor.$\\ And this is the choice for $r_1$.
As $S_{2}<0  \text{ if } \psi(S_{1}) + r_{2}\frac{\mu}{\mu -1} \leq \theta'$
we choose
$$r_2:= \max_{r \geq 0} \left\{  \psi(S_{1}) + r \frac{\mu}{\mu -1} \leq \theta'\right\} \Leftrightarrow r_2:= \left\lfloor \frac{\mu -1}{\mu} ( \theta' - \psi(S_{1}))\right\rfloor,$$
and so on. We notice that $r_{k}$ has a distinct expression.
Since $S_{k}<0  \text{ if } -1+ \psi(S_{k-1}) + r_{k}\frac{\mu}{\mu -1} \leq \theta'$, we choose
$$r_k:= \max_{r \geq 0} \left\{ -1+ \psi(S_{k-1})  + r \frac{\mu}{\mu -1} \leq \theta'\right\} \Leftrightarrow r_k:= \left\lfloor \frac{\mu -1}{\mu} ( \theta' +1 - \psi(S_{k-1}))\right\rfloor.$$

\begin{definition}\label{def: Laplacian classic shaerer sequence}
   Given $\mu \geq 4.38+$, let $\theta':= \frac{-(\mu -2) +\sqrt{(\mu -2)^2 - 4}}{2}$. For $\psi(t)=2-\mu-\frac{1}{t}$, we define the classic Laplacian Sharer's sequence as the sequence of caterpillars $G_{k}(\mu):=[r_{1}, r_{2}, ..., r_{k}], \; k \geq 2$ where
   \begin{equation}\label{eq:defclass}
   \left\{
     \begin{array}{ll}
       r_1:= \left\lfloor \frac{\mu -1}{\mu} ( \theta' - (1-\mu))\right\rfloor. \\
       S_{1}=1-\mu + r_{1} \frac{\mu}{\mu -1}\\
       r_{j}:= \left\lfloor \frac{\mu -1}{\mu} ( \theta' - \psi(S_{j-1}))\right\rfloor\\
       S_{j}=\psi(S_{j-1}) + r_{j} \frac{\mu}{\mu -1}\\
       r_{k}:= \left\lfloor \frac{\mu -1}{\mu} ( \theta' +1 - \psi(S_{k-1}))\right\rfloor\\
       S_{k}=-1+ \psi(S_{k-1}) + r_{k} \frac{\mu}{\mu -1}.
     \end{array}
   \right.
   \end{equation}
\end{definition}
\begin{example}\label{ex:Lap5.4} Consider applying Definition~\ref{def: Laplacian classic shaerer sequence} for $\mu=5.4$.  We have   $\theta:= \frac{-(\mu -2) -\sqrt{(\mu -2)^2 - 4}}{2} \approx -3.074772708$ and $\theta':= \frac{-(\mu -2) +\sqrt{(\mu -2)^2 - 4}}{2}= \theta^{-1}\approx -0.32522729>\theta$. We obtain by computation the first 11 terms, as follows.

\noindent $r_1:= \left\lfloor \frac{\mu -1}{\mu} ( \theta' - (1-\mu))\right\rfloor$, gives $r_1= 3$ \\
$S_{1}=1-\mu + r_{1} \frac{\mu}{\mu -1}$, gives $S_1 \approx -0.718181818$\\
$r_{2}:= \left\lfloor \frac{\mu -1}{\mu} ( \theta' - \psi(S_{1}))\right\rfloor$ , gives $r_2=1$ \\
$r_{j}:= \left\lfloor \frac{\mu -1}{\mu} ( \theta' - \psi(S_{j-1}))\right\rfloor$ , gives $r_j=1$ for $j=2,\ldots,10$, while\\
$r_{11}:= \left\lfloor \frac{\mu -1}{\mu} ( \theta' +1 - \psi(S_{10}))\right\rfloor$, where $S_{10} \approx -1.503801894,$ giving $r_{11} =2$.\\

Hence for $k=11$ our $G_k(5.4)$ is given by $$G_k:=[[1,1,1],[1],[1],[1],[1],[1],[1],[1],[1],[1],[1,1]].$$
\end{example}

Before we proceed with the analysis of this process, it is of great importance to understand how it evolves, in particular we want to obtain a relation between the sequence $S_{j}$  for $G_\mu(\mu):=[r_{1}, r_{2}, ..., r_{k-1}, r_{k}]$ comparatively with the one for $G_{k+1}(\mu)$. Let us denote it by $\tilde{S}_{j}$ and $G_{k+1}(\mu):=[\tilde{r}_{1}, \tilde{r}_{2}, ..., \tilde{r}_{k-1}, \tilde{r}_{k}, \tilde{r}_{k+1}]$.

We claim that, if $r_k \geq 1$ then $$\tilde{r}_k =r_k-1.$$

To prove this, we first notice that, by construction $r_{j}=\tilde{r}_{j}$ and $S_{j}=\tilde{S}_{j}$ for $1 \leq j \leq k-1$. We also recall that
\begin{equation}\label{eq:rk}r_k:= \max_{r \geq 0} \left\{ -1+ \psi(S_{k-1})  + r \frac{\mu}{\mu -1} \leq \theta'\right\}.
\end{equation}
We notice that this is equivalent to
\begin{equation}\label{eq:equiv1} \theta' -  \frac{\mu}{\mu-1} < -1 + \psi(S_{k-1})  + r_k \frac{\mu}{\mu -1} \leq \theta'.
\end{equation}
Equation~\eqref{eq:equiv1} is equivalent to
$$ -1+ \psi(S_{k-1})  + r_{k} \frac{\mu}{\mu -1} \leq \theta' \text{ and } -1+ \psi(S_{k-1})  + (r_{k}+1) \frac{\mu}{\mu -1} > \theta',$$
which are equivalent to
$$ \psi(S_{k-1})  + (r_{k} -1) \frac{\mu}{\mu -1} + \frac{1}{\mu-1} \leq \theta' \text{ and }  \psi(S_{k-1})  + r_{k}\frac{\mu}{\mu -1}+ \frac{1}{\mu-1}  > \theta'.$$
Rewriting these, we have
$$\psi(S_{k-1})  + (r_{k}-1) \frac{\mu}{\mu -1} \leq \theta' - \frac{1}{\mu-1}, \text{ and }  \psi(S_{k-1})  + r_{k} \frac{\mu}{\mu -1} > \theta' -\frac{1}{\mu-1} .$$
We claim that these inequalities imply
\begin{equation}\label{eq:equiv2}
\psi(S_{k-1})  + (r_{k}-1) \frac{\mu}{\mu -1} \leq \theta', \text{ and }  \psi(S_{k-1})  + r_{k} \frac{\mu}{\mu -1} > \theta'.\end{equation}
Indeed, as $ \theta-\frac{1}{\mu-1} \leq \theta'$, the first inequality follows. Now if $\psi(S_{k-1})  + r_{k} \frac{\mu}{\mu -1} \leq \theta'$, the first inequality implies that $\frac{\mu}{\mu-1} \leq 0$, which is a contradiction. It follows that, by comparing Equation~\eqref{eq:equiv2} with Equation~\eqref{eq:equiv1} we can write
\begin{equation}\label{eq:equiv3}r_k-1=\max_{r \geq 0} \left\{ \psi(S_{k-1})  + r \frac{\mu}{\mu -1} \leq \theta'\right\}.\end{equation}
Now we notice that, because $\psi(\tilde{S}_{k-1})=\psi(S_{k-1})$, the expression for $\tilde{r}_k$ is
\begin{equation}\label{eq: iquality r_k Laplacian}
  \tilde{r}_k:= \max_{r \geq 0} \left\{ \psi(S_{k-1})  + r \frac{\mu}{\mu -1} \leq \theta'\right\}= r_{k} -1,
\end{equation}
in view of Equation~\eqref{eq:equiv3}.

Analogously to the adjacency case, the procedure given in Definition~\ref{def: Laplacian classic shaerer sequence}, always works. In the sense that it produces a sequence of caterpillars $$G_{k}(\mu):=[r_{1}, r_{2}, ..., r_{k}], \; k \geq 2$$ such that the spectral radius $\rho_{L}(G_{k}(\mu)):= \rho_{L}(G_{k})$ has a limit point smaller than $\mu$. More precisely,  the sequence $\rho_{L}(G_{k})$ is increasing and bounded by $\mu$. Indeed, the part $\rho_{L}(G_{k}) < \mu$ is a trivial consequence of the construction because $S_1,\ldots, S_k <\theta'<0$ then the spectral radius for the Laplacian matrix of $G_{k}$ is necessarily smaller than $\mu$ by Lemma~\eqref{lem:detmu}. The part $\rho_{L}(G_{k}) < \rho_{L}(G_{k+1})$ is not quiet obvious because the choice of $r_k$ for $G_{k}$ is possibly different from the one for $G_{k+1}$, unless $r_{k}=0$. In that case $G_{k}$ is necessarily a subgraph of $G_{k+1}$ and $\rho_{L}(G_{k}) < \rho_{L}(G_{k+1})$. Thus, we assume that $r_{k}\geq 1$.  In order to verify the inequality $\rho_{L}(G_{k}) < \rho_{L}(G_{k+1})$, we recall that, from Equation~\eqref{eq: iquality r_k Laplacian}, we have $\tilde{r}_k = r_{k} -1$ meaning that $G_{k}$ is necessarily a subgraph of $G_{k+1}$ and then $\rho_{L}(G_{k}) < \rho_{L}(G_{k+1})$ again. In particular there exists
$\displaystyle\lim_{k \to \infty} \rho_{L}(G_k)=\gamma(\mu) \leq \mu$. We summarize the findings of this section, stating the following result.

\begin{teore}\label{thm:classLap} Let $\mu \in  [4.38+,\infty)$. The sequence $G_k(\mu)$ of caterpillars given by Equation~\eqref{eq:defclass} produces a limit point of $L$-spectral radii: there exists
$$\displaystyle\lim_{k \to \infty} \rho_{L}(G_k(\mu))=\gamma(\mu) \leq \mu.$$
\end{teore}

The question is: what are the limit points produced by this process?

\begin{definition}\label{def:classic laplacian shearer limit points}
   We denote by $\Omega_{1} \subseteq [4.38+, \infty)$ the set of all Laplacian limit points produced by Shearer's method
   $$\Omega_{1}:=\{ \gamma \; | \; \lim_{k \to \infty} \rho_{L}(G_k(\mu))=\gamma\}.$$
\end{definition}

Similar to the adjacency case studied by Shearer~\cite{shearer1989distribution}, where he shows that this method produces an $A$-limit point for any real number in $[\sqrt{2+\sqrt{5}}, \infty)$, it is natural to conjecture that $\Omega_{1} = [4.38+, \infty)$.

Unlike the adjacency case, however, we show in the sequel that there are points where $\gamma(\mu) < \mu$. In fact, we will provide a whole interval $I$, for which the sequence of caterpillars $G_k(\mu)$, given by Definition~\ref{def: Laplacian classic shaerer sequence} is the same for every $\mu \in I$. This means $G_k(\mu)$ has Laplacian spectral radius converging to the same value $\gamma < \mu$.

This only shows that the classic Shearer's caterpillar construction for the Laplacian version is not powerful enough to prove the density of Laplacian spectral radii in $[4.38+, \infty)$. In subsequent sections of this note, we extend this method and show how this extension obtains more accumulation points.

\subsection{A nasty interval} \label{sec:bad}
In this section we study the spectral radius of the caterpillar $G_k,~~k\geq 3$ given by $$G_k:=[[1,1,1],[1],[1],\ldots, [1],[1,1]].$$
We notice this was obtained in Example~\ref{ex: random gen shearer 5.4} as the classic Laplacian Shearer's construction with $\mu=5.4$ and $k=11$. We first consider applying {\rm Diagonalize}$(G_k,-\mu)$, for $\mu > 4.38+$,  using the notation established in Section~\ref{sec:gennotation}, disregarding the fact that this  sequence is generated by some  actual $\mu$.

Since the starlike trees $T_i$ are actually stars whose number of rays are $r_1=3, r_j=1, ~2 \leq j \leq k-1$, and $r_k=2$, the application of Equation~\eqref{eq:Laplacian psi} gives
$S_{1}= 1-\mu + 3 \frac{\mu }{\mu  -1},~~ S_{j}= 2-\mu - \frac{1}{S_{j-1}} + \frac{\mu }{\mu  -1}, ~2 \leq j \leq k-1$ and $S_{k}=1-\mu - \frac{1}{S_{k-1}} + 2\frac{\mu}{\mu -1}$. For convenience we let $\tilde{\psi}(t)=(2-\mu+\frac{\mu}{\mu-1})-\frac{1}{t}$, and rewrite as the following recurrence.
\begin{equation}\label{eq:badex}
\left\{\begin{array}{ll}
         S_{1}= 1-\mu + 3 \frac{\mu }{\mu  -1},\\
         S_{j}= \tilde{\psi}(S_{j-1}),~2 \leq j \leq k-1,\\
         S_k = \tilde{\psi}(S_{k-1})+\frac{1}{\mu-1}
       \end{array}
\right.
\end{equation}

The fixed points $\psi(t)=t$ are $\theta$ and $\theta'$ given by Equation~\eqref{eq:theta}, while the fixed points of $\tilde{\psi}(t)=t$ are given by
\begin{equation}\label{eq:sigma}
\left\{\begin{array}{ll}
  \sigma=\frac{(2- \mu+ \frac{\mu}{\mu -1})- \sqrt{(2- \mu+ \frac{\mu}{\mu -1})^2-4}}{2}\\
  \sigma'=\frac{(2- \mu+ \frac{\mu}{\mu -1})+ \sqrt{(2- \mu+ \frac{\mu}{\mu -1})^2-4}}{2}=\sigma^{-1}.
\end{array}
\right.
\end{equation}

\begin{lemma}\label{lem:u*} Let $4<\mu<6 $. Let $\mu_*$ and $\mu^*$ be, respectively, the solutions of $S_1=\sigma'$ and $\theta^{-1}-\frac {\mu}{\mu-1} =\sigma$.
\begin{itemize}
  \item[$(a)$] $\mu_*=\frac{5+\sqrt{33}}{2}=5.3722+$ is the largest root of  the polynomial $x^2-5x-2$;
  \item[$(b)$] $\mu^*=5.4207+$ is the largest root of $2\,{x}^{4}-14\,{x}^{3}+19\,{x}^{2}-10\,x-1.$
\end{itemize}
\end{lemma}
\begin{proof}
$(a)$ $~S_1=\sigma'$ is equivalent to
$$1-\mu + 3 \frac{\mu}{\mu -1}= \frac{(2- \mu+ \frac{\mu}{\mu -1})+ \sqrt{(2- \mu+ \frac{\mu}{\mu -1})^2-4}}{2},$$
which can be transformed into the  polynomial $\mu^2-5\mu-2$. In the range $4<\mu<6$, $\mu_{*}$ is the only root of the polynomial.\\
$(b)$ ~Now, for $\theta^{-1}-\frac {\mu}{\mu-1} =\sigma$ we can perform the same computations transforming
$$\frac{(2- \mu)+ \sqrt{(2- \mu)^2-4}}{2}-\frac {\mu}{\mu-1} = \frac{(2- \mu+ \frac{\mu}{\mu -1})- \sqrt{(2- \mu+ \frac{\mu}{\mu -1})^2-4}}{2}$$
into a polynomial equation as prescribed. Performing that we obtain that $\mu^{*}$ is the largest root of $2\,{\mu}^{4}-14\,{\mu}^{3}+19\,{\mu}^{2}-10\,\mu-1.$

 \end{proof}

\begin{lemma}\label{lem:convex} Let $G_k$ be the caterpillar $G_k:=[[1,1,1],[1],[1],\ldots, [1],[1,1]]$ and $\mu_k=\rho_L(G_k)$ be the spectral radius of $G_k$. Then
$$\displaystyle\lim_{k \to \infty} \rho_{L}(G_k)=\mu_{*}.$$
\end{lemma}
\begin{proof} We consider the auxiliary graphs $H_{k}(\mu):=[[1,1,1],[1],\ldots,[1],[1]], \; k \geq 3$. By applying  ${\rm Diagonalize}(H_{k}, -\mu_{k})$ with $\mu_{k}=\rho_{L}(G_{k})$, we obtain the sequence of values $W_{1},\ldots, W_{k-1}, W_{k}$,
$$W_{1}= 1-\mu_{k} + 3 \frac{\mu_{k}}{\mu_{k} -1} <0$$
$$W_{j+1}:= \tilde{\psi}(W_{j}),\; 1 \leq j,$$
where $\tilde{\psi}(t):= (2- \mu_{k}+ \frac{\mu_{k}}{\mu_{k} -1}) -\frac{1}{t}, \; t \neq 0$. From \cite{OliveTrevAppDiff} we know that a closed formula $W_j$ is as follows.
$$W_{j}:=\sigma_{k}  +\frac{\sigma_{k}^{-1} -\sigma_{k}}{\beta_{k}\sigma_{k}^{2j -2}+1},\; 1 \leq j,$$
where $\beta_{k}:=\frac{\sigma_{k}^{-1} -\sigma_{k}}{W_{1} -\sigma_{k}} -1$ and $\sigma_{k}=\frac{(2- \mu_{k}+ \frac{\mu_{k}}{\mu_{k} -1})- \sqrt{(2- \mu_{k}+ \frac{\mu_{k}}{\mu_{k} -1})^2-4}}{2}$. Since $H_{k}$ and $G_{k}$ only differ in the last term, applying ${\rm Diagonalize}(G_{k}, -\mu_{k})$, we obtain a sequence of values $S_{1}=W_{1},\ldots, S_{k-1}=W_{k-1}$ and $S_{k}= W_{k} + \frac{1}{\mu_{k} -1} =0$  because $\mu_{k}:=\rho_{L}(G_{k})$. An easy computation transforming this equation into a polynomial with respect to $\sigma_{k}$ shows that taking limit when $k\to\infty$ in $W_{k} + \frac{1}{\mu_{k} -1} =0$, we obtain that the limit point $\mu:=\lim_{k \to \infty} \rho_{L}(G_k )$ satisfies $(\sigma + \frac{1}{\mu -1})\cdot(S_{1} - \sigma^{-1})=0$.

Indeed, $$S_{k}= W_{k} + \frac{1}{\mu_{k} -1} =0 \Leftrightarrow W_{k} + \frac{1}{\mu_{k} -1} =0$$
$$\sigma_{k}  +\frac{\sigma_{k}^{-1} -\sigma_{k}}{\beta_{k}\sigma_{k}^{2j -2}+1} + \frac{1}{\mu_{k} -1} =0$$
$$\sigma_{k}  +\frac{\sigma_{k}^{-1} -\sigma_{k}}{ \left(\frac{\sigma_{k}^{-1} -\sigma_{k}}{W_{1} -\sigma_{k}} -1\right)\sigma_{k}^{2j -2}+1} + \frac{1}{\mu_{k} -1} =0.$$
Now we just transform this into a polynomial and use that $\sigma_{k}^{2j -2} \to \infty$. The remaining equation is $(\sigma_k + \frac{1}{\mu -1})\cdot (W_{1} - \sigma_k^{-1}) =0$. As $W_{1}=S_{1}$ and since $\sigma_k + \frac{1}{\mu -1}$ has no solution for $\mu>4$, we obtain the equation $S_{1} = \sigma_k^{-1}$, which, by Lemma~\ref{lem:u*} is $\mu_*$.
\end{proof}

\begin{lemma}\label{lem:interv} Let $\mu_*=5.3722+$ and $\mu^*=5.4207+$ be as in Lemma~\ref{lem:u*} and let $I=[\mu_*,\mu^*]$. Then for all $\mu \in I$, the classic Laplacian Shearer's  sequence given by Definition~\ref{def: Laplacian classic shaerer sequence} is $G_k(\mu):=[[1,1,1],[1],[1],\ldots, [1],[1,1]].$
\end{lemma}
\begin{proof}
We consider two facts, the first one is
$$\theta^{-1}-\frac {\mu}{\mu-1}<1-\mu+3\frac{\mu}{\mu-1} <\theta^{-1}$$
for $ 5.09807+ < \mu < 6.05505+$ (see Figure~\ref{fig:ineq-example-05}), meaning that in this interval every $G_{k}$ begins with $[[1,1,1], ?,?,\ldots]$ because $I \subset [5.09807+ , 6.05505+]$.
\begin{figure}[ht!]
  \centering
  \includegraphics[width=5cm]{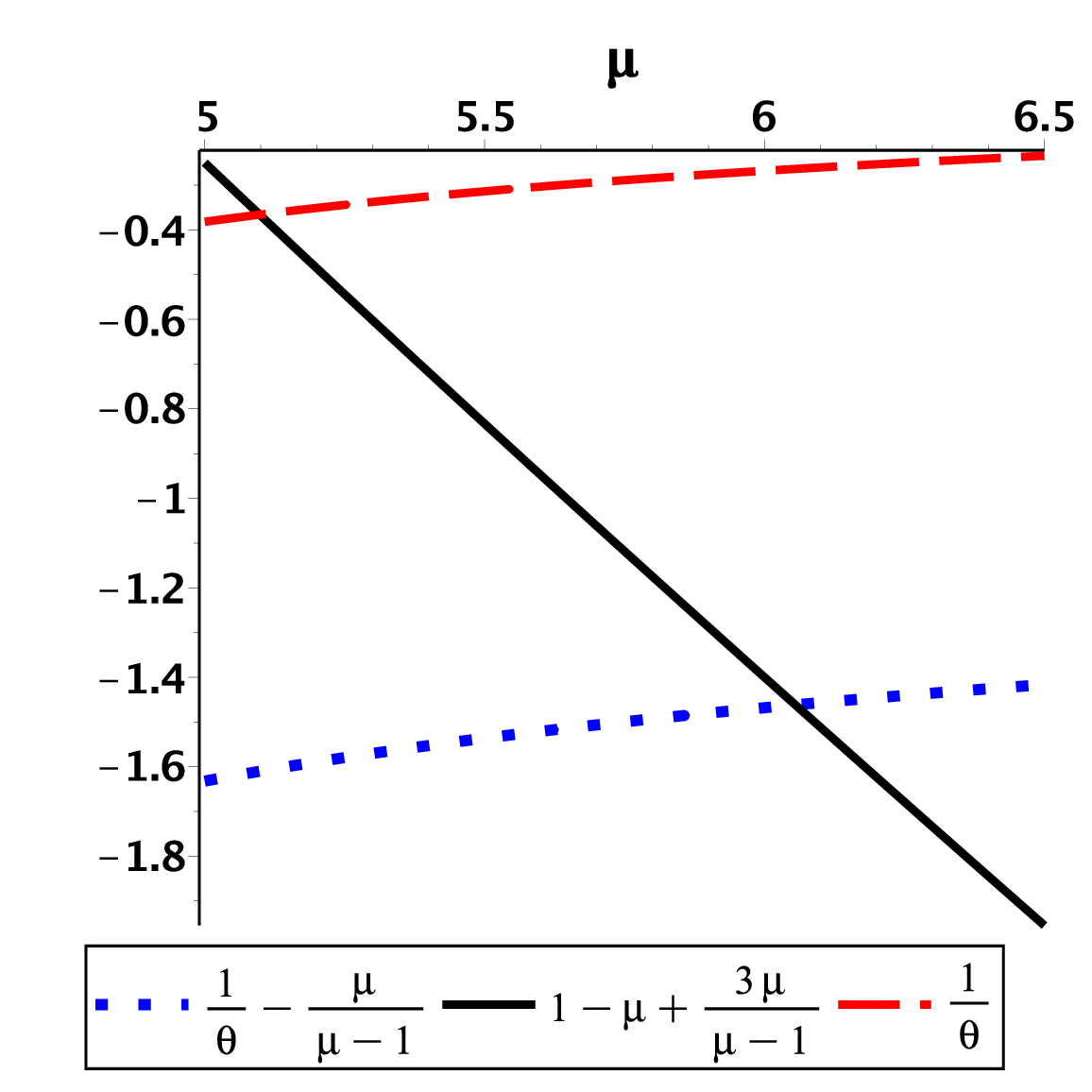}
  \caption{Inequality $\theta^{-1}-\frac {\mu}{\mu-1} <1-\mu+3\frac {\mu}{\mu-1} <\theta^{-1}.$}\label{fig:ineq-example-05}
\end{figure}
By applying  ${\rm Diagonalize}(G_{k}, -\mu)$ with $\mu \in I$, we obtain a sequence of values $S_{1}, ..., S_{k-1}, S_{k}$ where
$S_{1}= 1-\mu + 3 \frac{\mu }{\mu  -1}  \in (\theta^{-1}-\frac {\mu}{\mu-1} , \theta^{-1}]$, that is $r_{1}=3$.

The second  fact is the inequality
$$\theta^{-1}-\frac {\mu}{\mu-1} <\sigma <1-\mu+3\frac {\mu}{\mu-1} <\sigma^{-1}<\theta^{-1},$$
where $\sigma$ and $\sigma^{-1}$  are defined in Equation~\eqref{eq:sigma}. This is easily verified from the graphics given in Figure~\ref{fig:ineq-example-01020304}, since those are functions of $\mu$.

\begin{figure}[ht!]
  \centering
  \includegraphics[width=4.5cm]{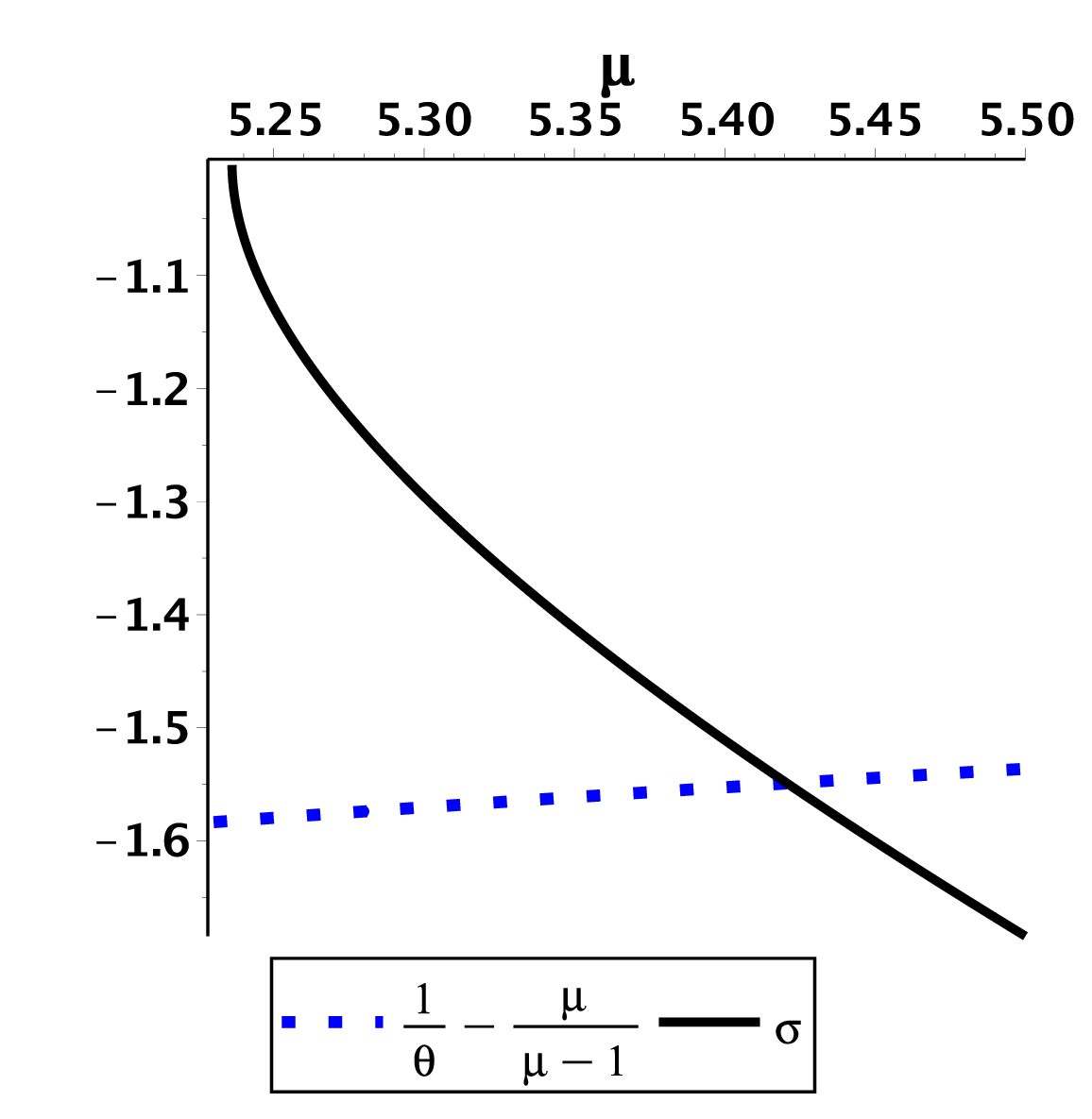}\hspace{1cm}
  \includegraphics[width=4.5cm]{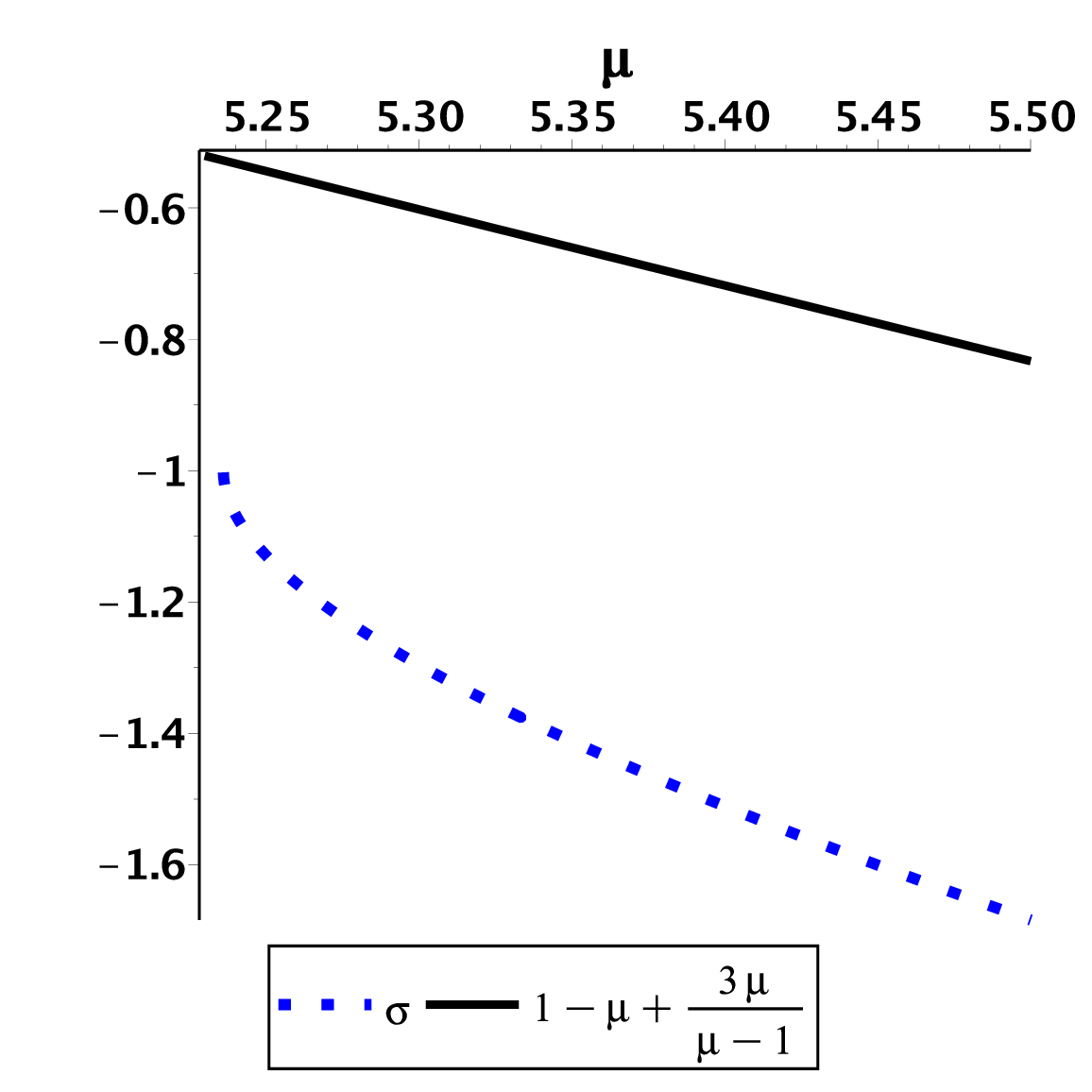}\\
  \includegraphics[width=4.5cm]{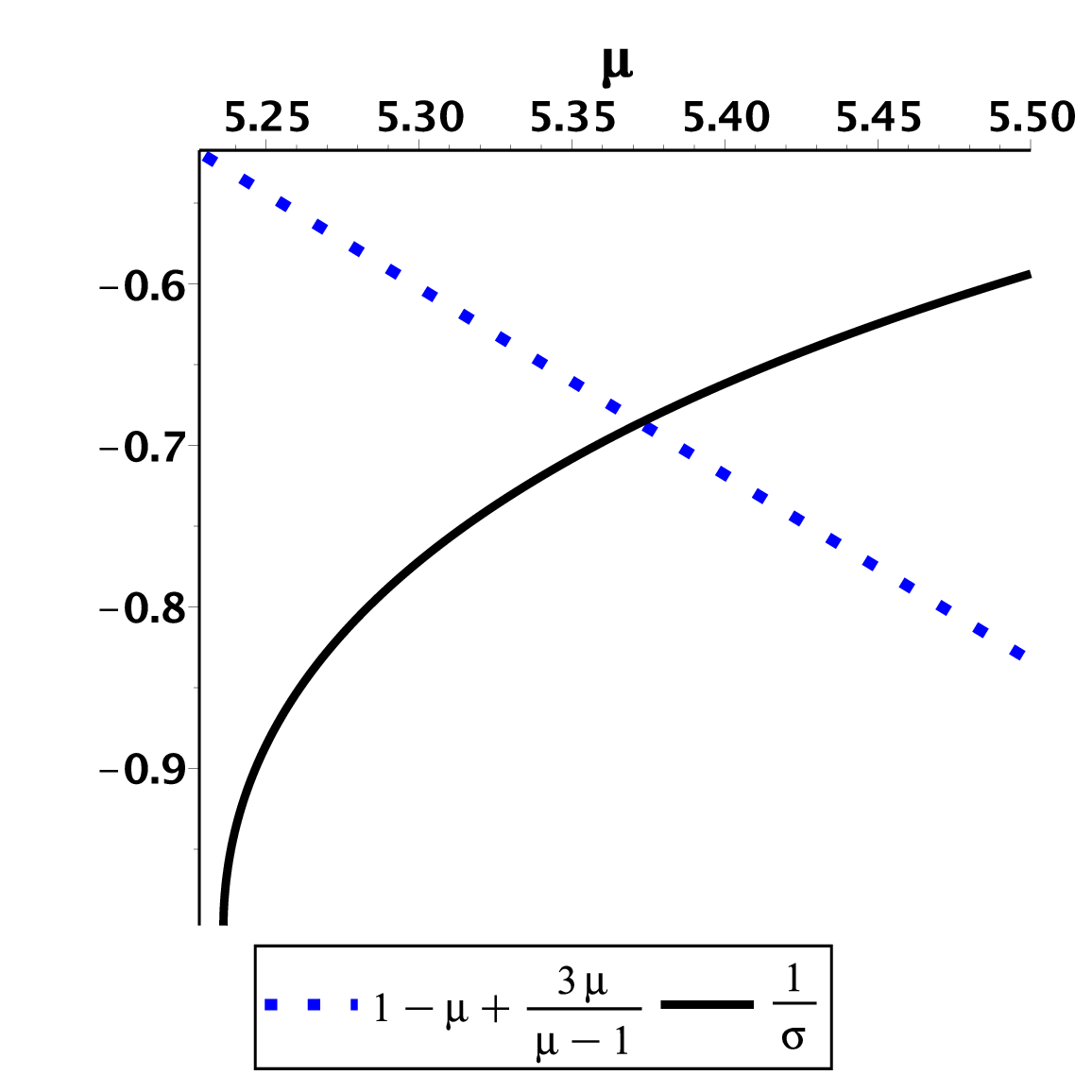}\hspace{1cm}
  \includegraphics[width=4.5cm]{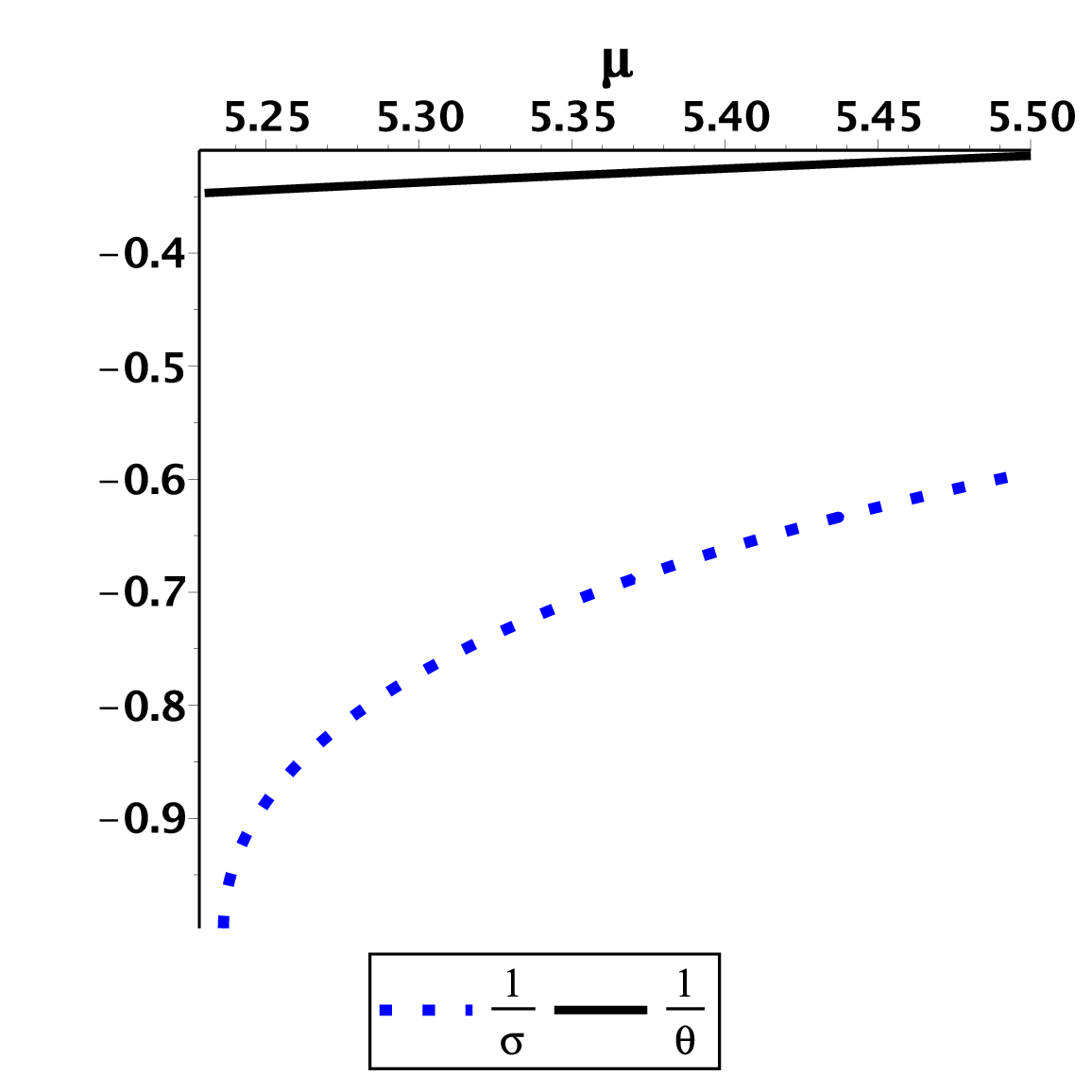}
  \caption{Inequality $\theta^{-1}-\frac {\mu}{\mu-1} <\sigma <1-\mu+3\frac {\mu}{\mu-1} <\sigma^{-1}<\theta^{-1}$.}\label{fig:ineq-example-01020304}
\end{figure}

We know that $S_{2}= 2-\mu - \frac{1}{S_{1}} + r_{2} \frac{\mu }{\mu  -1}$ and we claim that $r_{2}=1$ because $2-\mu - \frac{1}{S_{1}} +  \frac{\mu }{\mu  -1} = \tilde{\psi}(S_{1}) \in (\sigma , \sigma^{-1}]\subset (\theta^{-1}-\frac {\mu}{\mu-1} , \theta^{-1}]$ meaning that $r_{2}=1$.
In each step, $S_{j+1}=\tilde{\psi}(S_{j})$ obeys the recurrence generated by $\tilde{\psi}(t):= (2- \mu + \frac{\mu }{\mu  -1}) -\frac{1}{t}, \; t \neq 0$. From \cite{OliveTrevAppDiff} we know that, since $\Delta=(2- \mu + \frac{\mu }{\mu  -1})^2-4>0$, we have
$$S_{j}:=\sigma   +\frac{\sigma^{-1} -\sigma}{\beta\sigma^{2j -2}+1},\; 1 \leq j \leq k-1,$$
where $\beta:=\frac{\sigma^{-1} -\sigma}{S_{1} -\sigma} -1$ and $\sigma=\frac{(2- \mu+ \frac{\mu}{\mu -1})- \sqrt{(2- \mu+ \frac{\mu}{\mu -1})^2-4}}{2}$. We notice that $\psi(S_{j}) \in (\sigma , \sigma^{-1}]\subset (\theta^{-1}-\frac {\mu}{\mu-1} , \theta^{-1}]$,  meaning that $r_{j}=1$ for $j \leq k-1$. We could extrapolate for $\psi(S_{k-1}) \in (\sigma , \sigma^{-1}]\subset (\theta^{-1}-\frac {\mu}{\mu-1} , \theta^{-1}]$ obtaining $r_{k}=1+1$.\\ Thus $G_{k}(\mu):=[[1,1,1],[1],...,[1],[1,1]], \; k \geq 3$, as claimed.

\end{proof}

\begin{remark}
  We are not claiming that the points $\mu \in I$  are not Laplacian limit points, only that  they can not be achieved by using caterpillar graphs according to the Shearer's process. This motivates us to consider a more general class of trees in the next section.
\end{remark}

\section{Sequences of linear trees}\label{sec:lintreeseq}
Inspired by the Shearer's procedure from~\cite{shearer1989distribution}, we will consider increasing sequences of Laplacian spectral radii of graphs $\{G_k \; | \; k \in \mathbb{N}\}$, where $G_k$ are linear trees, as introduced in Section~\ref{sec:gennotation}.

We need some technical results. The first one is:
\begin{lemma}\cite{Anderson-1985, zhang2011laplacian}\label{lem:bound}
Let $G=(V,E)$ be a graph on $n$ vertices and $m$ edges.  Let $V=\{v_1,v_2,\ldots,v_n\}$ and let $d(v_i)$ be the degree of the vertex $v_i$. Then
$$\rho_{L}(G) \leq \max \left\{d(v_i)+d(v_j) ~|~ v_i, v_j \in E \right\}.$$
In particular, we can say that $\rho_{L}(G) \le 2 \Delta(G)$,  where $\Delta(G)$ is the largest degree of $G$.
%$$\min \left\{d(v_i)+d(v_j) ~|~ v_i, v_j \in E \right\} \leq \mu(G) \leq \max \left\{d(v_i)+d(v_j) ~|~ v_i, v_j \in E \right\}$$
\end{lemma}
The second one is a bound on the spectral radii of $\{G_k \; | \; k \in \mathbb{N}\}$. We recall that $\omega(T_j)$ is the number of paths that compose the starlike $T_j$.
\begin{lemma}\label{lem:boundness seq starlikes}
Consider a sequence of linear trees $\{G_{k} \; | \; k \in \mathbb{N}\}$. Then,
$$
\exists M>0, \; \rho_{L}(G_{k}) \leq M, \forall  k\in\mathbb{N} \Leftrightarrow \max_{k} \max_{1 \leq j \leq k}  \omega(T_{j}) <+\infty.$$
\end{lemma}
\begin{proof}
Assume that $\exists M>0, \; \rho_{L}(G_{k}) \leq M, \forall k\in\mathbb{N}$, then necessarily ${\displaystyle \max_{k} \max_{1 \leq j \leq k}  \omega(T_{j}) <+\infty}$ otherwise if there exists a sequence of $T_j$ with  $\omega(T_{j}) \to \infty$ then, since $\rho_{L}(G_{k}) \geq 1 + \omega(T_{j})$, would be a contradiction.
%\cvt{Elismar, I understand  we show that $\rho_{L}(F_{k}) < 5+r_0$. Seems that we don't need this, unless this bound is used latter. See this argument:\\
%From this result...
%  Hence, can argue like this. If you agree, we can delete this red comment and the long proof after this:}

Reciprocally, let  ${\displaystyle r_0:=\max_{k} \max_{1 \leq j \leq k}  \omega(T_{j}) <+\infty}$ and ${\displaystyle h_0:=\max_{1 \leq j \leq k}  h(T_{j})}$ then for any $k$, $G_{k}$ is a subgraph of
$F_{k}=[[h_0,\ldots,h_0], \ldots,[h_0,\ldots,h_0]],$
where $w([h_0,\ldots,h_0])=r_{0}$ and $\ell(F_{k})=k$. Thus, $\rho_{L}(G_{k}) \leq \rho_{L}(F_{k})$, for any $k$.

A well known upper bound for the $\rho_{L}(F_k)$ is $\rho_{L}(F_k) \le 2\Delta(F_k)$ (see Lemma~\ref{lem:bound}),  where $\Delta(F_k)$ is the largest degree of $F_k$. As the $\Delta(F_k) \leq r_0 + 2$, we have  $\rho_{L}(G_{k}) \leq \rho_{L}(F_{k}) \leq 2( r_0 + 2)$ and the result follows.

\end{proof}

\begin{definition}\label{def:sequence_starlike_limited}
   We denote by $\mathcal{L}_{\infty}$ the set of all sequences  of starlike trees with bounded width, that is,
   $$\mathcal{L}_{\infty}:=\left\{ (T_{1},T_{2},...)\; | \; T_j's \text{ are starlike trees and } \max_{k \geq 1} \max_{1 \leq j \leq k}  \omega(T_{j}) <+\infty\right\}.$$
\end{definition}

\begin{definition}\label{def:generalized shearer sequence}
   For each $\mathbb{T}=(T_{1},T_{2},...) \in \mathcal{L}_{\infty}$ and $C=(C_{2},C_{3},...) \in \mathcal{L}_{\infty}$ we can define a sequence of linear trees
   \[
   \left\{
     \begin{array}{ll}
       G_{1}(\mathbb{T}):=[T_{1}] \\
       G_{k}(\mathbb{T}):=[T_{1},T_{2},..., T_{k-1}, C_{k}], \; k \geq 2,
     \end{array}
   \right.
   \]
called (Laplacian) \emph{generalized Shearer sequence} associated to $\mathbb{T}$ and $C$ if $\rho_{L}(G_{k}(\mathbb{T})) < \rho_{L}(G_{k+1}(\mathbb{T}))$ for all $k \geq 1$.
\end{definition}

We actually should use the notation $G_{j}(\mathbb{T}, C)$, but we will omit the dependence in $C$ when it is clear.  A sufficient condition to have  $\rho_{L}(G_{k}(\mathbb{T})) < \rho_{L}(G_{k+1}(\mathbb{T}))$ is  $G_{k}(\mathbb{T})$  to be a subgraph of $G_{k+1}(\mathbb{T})$  for all $k \geq 1$, called subgraph condition.

\begin{remark}\label{rem:canonical sequence}
   The canonical case we may deal is, $\mathbb{T}=(T_1, T_2,...) \in \mathcal{L}_{\infty}$ and  $C=(T_2,T_3,...)$ a shift of $\mathbb{T}$ (thus $C \in \mathcal{L}_{\infty}$), then $G_1=[T_1]$, $G_2=[T_1, T_2]$ and $G_3=[T_1, T_2, T_3]$, and so on. This sequence naturally satisfies the subgraph condition so it is a generalized Shearer sequence.
\end{remark}

\begin{example}\label{ex:gen shearer seq}
Let us consider some cases for generalized Shearer sequences:\\
  \begin{itemize}
    \item[a)] Consider  $\mathbb{T}=([1,1,1],[1],[1],[1],[1],...) \in \mathcal{L}_{\infty}$ and\\ $C=([1,1],[1,1],[1,1],[1,1],...) \in \mathcal{L}_{\infty}$, then\\
        $G_1=[[1,1,1]]$, $G_2=[[1,1,1],[1,1]]$, $G_3=[[1,1,1],[1],[1,1]]$,\\ $G_4=[[1,1,1],[1],[1],[1,1]]$, and so on, is obviously a generalized Shearer sequence coincident with the one in Section~\ref{sec:bad}.\\
    \item[b)] Consider  $\mathbb{T}=([1,1],[1],[0],[1],[0],...) \in \mathcal{L}_{\infty}$ and $C=([0],[0],[0],[0],...) \in \mathcal{L}_{\infty}$ then $G_1=[[1,1]]$, $G_2=[[1,1],[0]]$, $G_3=[[1,1],[1],[0]]$,  $G_4=[[1,1],[1],[0],[0]]$,\\ $G_5=[1,1],[1],[0],[1],[0]$, and so on, is  a generalized Shearer sequence.\\
    \item[c)]  Consider  $\mathbb{T}=([1,1],[1,1],[1],[1],[1],...) \in \mathcal{L}_{\infty}$ and\\ $C=([1,1,1,1],[0],[0],[0],...) \in \mathcal{L}_{\infty}$ then $G_1=[[1,1]]$, $G_2=[[1,1],[1,1,1,1]]$ and $G_3=[[1,1],[1,1],[0]]$. Note that $G_2$  is not a subgraph of $G_3$, moreover,
        $$\rho_{L}(G_{2})=6.141336+ >  5.261802+=\rho_{L}(G_{3})$$
          thus it is not a generalized Shearer sequence.
    \item[d)]  Consider  $\mathbb{T}=([1],[1,1],[1,1,1],[1,1,1,1],...) $ and\\ $C=([1],[1],[1],[1],...) \in \mathcal{L}_{\infty}$ then $G_1=[[1]]$, $G_2=[[1],[1]]$ and $G_3=[[1],[1,1],[1]]$, $G_4=[[1],[1,1],[1,1,1], [1]]$, and so on. This sequence satisfies the subgraph condition however it is not a generalized Shearer sequence because $\omega(T_j)=j-1 \to \infty$, that is, $C\in \mathcal{L}_{\infty}$ but $\mathbb{T} \not\in \mathcal{L}_{\infty}$.
  \end{itemize}
\end{example}

By construction, each generalized Shearer sequence satisfies the property that $\rho_{L}(G_{k}) \leq \rho_{L}(G_{k+1}) $ and by Lemma~\ref{lem:boundness seq starlikes} we also know that $k \to \rho_{L}(G_{k})$ is bounded. Thus, there exists $\displaystyle\lim_{k \to \infty} \rho_{L}(G_{k}):= \gamma(\mathbb{T})$.

\begin{definition}\label{def:generalized shearer sequence limits}
   We denote by $\mathcal{S}$ the set of all limit points of (Laplacian) generalized Shearer sequences, that is,
   $$ \mathcal{S}:=\{ \gamma \; | \; \gamma=\gamma(\mathbb{T}) \text{ for some } \mathbb{T}, C \in \mathcal{L}_{\infty}\}.$$
\end{definition}

\section{Applications: Generalized random Shearer process}\label{sec:appProcess}
As the set $\Omega_1 \subseteq \mathcal{S}\subseteq [4.38+, \infty)$ (see Section~\ref{sec:revisit shearer Laplacian}), it is natural to believe that $\mathcal{S}$ is  a larger set of Laplacian limit points. Indeed, we will show this. Neverthless, it will be useful to have a structured way to construct linear trees as generalized Shearer sequences, for a given parameter $\mu$. The main goal of this section is provide one, which is an adaptation of the classical process used by Shearer to produce caterpillars. We denominate it as  \textbf{Generalized random Shearer process}.

Fixed $\mu > 4.38+$, we will construct a generalized Shearer sequence in the following way. Let $\mathbb{T}=(T_{1},T_{2},\ldots) $ and $C=(C_{2},C_{3},\ldots) $ be arbitrary sequences of starlike trees. From that we define linear trees
$$G_{k}:=[T_{1},T_{2},, T_{k-1}, C_{k}], \; k \geq 2,$$
with $G_{1}:=[T_{1}]$. We would like to have $\rho_{L}(G_{k}) < \mu$ and $G_{k}$ a subgraph of $G_{k+1}$.

To achieve this property, we use ${\rm Diagonalize }(G_{k}, -\mu)$ obtaining $\Pi(G_{k}, \mu)=(S_{1}, \ldots, S_{k})$ (see Section~\ref{sec:gennotation}).  We recall that $\psi(t)=2-\mu -\frac{1}{t}, \; t\neq 0$, $\theta:= \frac{-(\mu -2) -\sqrt{(\mu -2)^2 - 4}}{2}$ is an attracting point and $\theta':= \frac{-(\mu -2) +\sqrt{(\mu -2)^2 - 4}}{2}= \theta^{-1}>\theta$ is a repelling one.

Thus, in order to obtain $S_{j}<0$ for all $j$  we successively choose the following starlike trees:
\[
\left\{
  \begin{array}{ll}
    T_1 \text{ such that }  1  -\mu +  \delta(T_{1},\mu) < \theta^{-1}; \\
    T_j \text{ such that } \psi(S_{j-1}) + \delta(T_{j},\mu) < \theta^{-1},\; 2 \leq j \leq k; \\
    C_k \text{ such that } C_k:=T_k.
  \end{array}
\right.
\]
Since  $\psi(S_{k-1}) + \delta(T_{k},\mu) < \theta^{-1}$, $\theta^{-1}<0<1$ and  $C_k=T_k$ we obtain
$$\psi(S_{k-1}) + \delta(T_{k},\mu) < 1$$
$$-1+\psi(S_{k-1}) +  \delta(C_{k},\mu) < 0$$
$$S_{k} <0.$$
Thus $\rho_{L}(G_{k}) < \mu$ because $\sign(\Pi)(G_{k}, \mu)=(-,\ldots, -)$.

This procedure is random since we can choose different trees in each step and it will always work because $ 1  -\mu  < \theta^{-1}$ so, in the worst case, we can choose $T_1=[0]$ and the same is true for the next ones since the previous $S_{j-1} < \theta^{-1}$ implies that  $\psi(S_{j-1}) < \theta^{-1}$, by construction.

We notice that necessarily $\mathbb{T}, C \in \mathcal{L}_{\infty}$ because for each choice of $T_j$ or $C_k$ we must have subgraphs of $G_k$ which has spectral radius smaller than $\mu$, a fixed number, this means that the width $\omega(T_j)$ or $\omega(C_k)$ can not exceed $\mu - 4$ (recall that $\rho_{L}(G_{k}) \geq \max\{\omega(T_j)+3,\omega(C_k)+2\})$.
%\cvt{Qual a lower bound usada aqui? Aqui precisa fazer um ajuste por causa da cota maior que usamos acima?}

From the way we built the $G_{k}$'s, we always end up with a generalized Shearer sequence since $C$ is a shift of $\mathbb{T}$. We then have $\rho_{L}(G_{k}) < \rho_{L}(G_{k+1})$ and, in particular, there exists
$$\lim_{k \to \infty} \rho_{L}(G_k)=\gamma(\mu) \leq \mu.$$

\begin{definition}\label{def:random laplacian shearer limit points}
   We denote by $\Omega \subseteq [4.38+, \infty)$ the set of all Laplacian limit points produced by the Generalized random Shearer's method
   $$\Omega:=\{ \gamma \; | \; \lim_{k \to \infty} \rho_{L}(G_k(\mu))=\gamma\}.$$
\end{definition}

\begin{remark}
  We notice that, if in addition to the above requirements, we restrict our choices to $\max\{h(T_j),h(C_k)\} =1$ and we maximize the value of each $S_j < \theta^{-1}$ we recover the caterpillars from the classic Shearer's paper. Thus $\Omega_{1} \subset \Omega$.
\end{remark}

\begin{example}\label{ex: random gen shearer 5.4}
   Consider $\mu=5.4$ and a random Shearer process running up computationally producing a generalized Shearer sequence
   $$
   \mathbb{T}:=([1], [1, 2], [1], [0], [0], [1], [1], [0], [1], [2], [2], [2], [1],[1], [2], [2], [2], [1], [2], [1],[2],
   $$
   $$
   [1], [2], [2], [1], [2], [1], [2], [2], [2], [2], [2], [2], [2], [1], [1], [2], [1], [2], [2], [1],[2], [2], [2],[1],
   $$
   $$
   [2], [2], [2], [2], [2], [1], [2], [2], [2], [2], [2], [2], [1], [2], [2],[1], [1], [2], [2], [2], [2], [2],[2], [2],
   $$
   $$
   [1], [2], [2], [1], [2],[1], [1], [2], [2], [2], [1], [1], [1], [1], [1], [1], [2], [2], [1], [1],[2], [2],[1], [2],
   $$
   $$
   [2],[1], [2], [1], [2], [2], [2], [0]^{\infty}).
   $$

By a direct computation we obtain \\
   $G_{1}:=[[1]]$ and $\rho_{L}(G_{1})=2$;\\
   $G_{2}:=[[1], [1, 2]]$ and $\rho_{L}(G_{2})=4.302775+$;\\
   $G_{3}:=[[1], [1, 2], [1]]$ and $\rho_{L}(G_{3})=5.236067+$;\\
   $G_{4}:=[[1], [1, 2], [1], [0]]$ and $\rho_{L}(G_{4})=5.3817984+$;\\
   $G_{5}:=[[1], [1, 2], [1], [0],[0]]$ and $\rho_{L}(G_{4})=5.397488988+$\\
and so on. For instance $\rho_{L}(G_{100})=5.39999504+$, proving that by continuing the process we obtain some Laplacian limit point in the interval $(\mu_{*}, \mu^{*})$. This proves that the generalized random Shearer process generates Laplacian limit points which are not achieved by classical approach using caterpillars, that is, $\Omega_{1} \subsetneq \Omega$.

It is worth mentioning that, for the same $\mu=5.4$, we could have made  different choices in each step, obtaining a very different sequence. We will compute the first terms using the maximum criteria (as the classical Shearer's method) and leave to the reader to check that these are optimal choices once we assume $h(T_j)\leq 2$. This, however, result that is worse in terms of approximation to $\mu=5.4$:\\
   Take $T_1:=[2,2,2]$ because $S_{1}=1-\mu-\delta([2,2,2],\mu) =-0.8843+ < -0.3252+=\theta^{-1}$;\\
   Take $T_2:=[0]$ because $S_{2}=\psi(S_{1})-\delta([0],\mu) =-1.1994+ < -0.3252+=\theta^{-1}$;\\
   Take $T_3:=[2]$ because $S_{3}=\psi(S_{2})-\delta([2],\mu) =-1.2511+ < -0.3252+=\theta^{-1}$;\\
   Take $T_4:=[2]$ because $S_{4}=\psi(S_{3})-\delta([2],\mu) =-1.2855+ < -0.3252+=\theta^{-1}$ and so on.\\
   Thus we have the sequence $\mathbb{T}:=([2,2,2],[0], [2], [2], [2]^{\infty}).$
   By a direct computation we obtain \\
   $G_{1}:=[[2,2,2]]$ and $\rho_{L}(G_{1})=4.4142+$\\
   $G_{2}:=[[2,2,2],[0]]$ and $\rho_{L}(G_{2})=5.2360+$\\
   $G_{3}:=[[2,2,2],[0], [2]]$ and $\rho_{L}(G_{3})=5.3105+$\\
   $G_{4}:=[[2,2,2],[0], [2], [2]]$ and $\rho_{L}(G_{4})=5.3325+$\\
   $G_{5}:=[[2,2,2],[0], [2], [2],[2]]$ and $\rho_{L}(G_{4})=5.3423+$.\\
   Finally, notice that $\rho_{L}(G_{100})=5.3554+ < \mu_{*}$ indicating that the limit does not reach the interval $(\mu_{*}, \mu^{*})$ as the previous one!
\end{example}

In some sense, a generalized Shearer sequence $G_{k}(\mathbb{T},C)$ is always the realization of a generalized random Shearer process, for a given $\mu$.
\begin{lemma}\label{lem:converse gen random}
   Consider the generalized Shearer sequence $G_{k}(\mathbb{T},C)$
   \[
   \left\{
     \begin{array}{ll}
       G_{1}(\mathbb{T}):=[T_{1}] \\
       G_{k}(\mathbb{T}):=[T_{1},T_{2},\ldots, T_{k-1}, C_{k}], \; k \geq 2,
     \end{array}
   \right.
   \]
   where  $\mathbb{T}=(T_{1},T_{2},\ldots) \in \mathcal{L}_{\infty}$ and $C=(C_{2},C_{3},\ldots) \in \mathcal{L}_{\infty}$. Let $\mu:=\displaystyle\lim_{k \to \infty} \rho_{L}(G_{k})$ and $\theta:= \frac{-(\mu -2) -\sqrt{(\mu -2)^2 - 4}}{2}$, then
   \[
\left\{
  \begin{array}{ll}
    S_{1}=1  -\mu +  \delta(T_{1},\mu) < \theta^{-1}; \\
    S_{j}=\psi(S_{j-1}) + \delta(T_{j},\mu) < \theta^{-1},\; 2 \leq j \leq k,
  \end{array}
\right.
\]
where $\Pi(G_{k}, \mu)=(S_{1},\ldots, S_{k})$ is obtained from ${\rm Diagonalize}(G_{k}, -\mu)$.
In particular, if $C$ is a shift of $\mathbb{T}$ then we have a generalized random Shearer process.
\end{lemma}
\begin{proof}
We know that  $\rho_{L}(G_{k})$ is strictly increasing, thus  $\rho_{L}(G_{k}) < \mu$. Applying ${\rm Diagonalize} (G_{k}, -\mu)$ we should obtain only negative values $S_{j}$. Let us suppose, by contradiction, that $S_{j_{0}}>\theta^{-1}$, for some $j_{0} \geq 1$, that is, $S_{j_{0}}\in (\theta^{-1}, 0)$, then
   $$S_{j_{0}+1}=\psi(S_{j_{0}}) + \delta(T_{j_{0}+1},\mu) \geq \psi(S_{j_{0}}) >  S_{j_{0}}.$$
   If $S_{j_{0}+1}>0$ we have a contradiction, otherwise $S_{j_{0}+1}\in (\theta^{-1}, 0)$ and
   $$S_{j_{0}+2}=\psi(S_{j_{0}+1}) + \delta(T_{j_{0}+2},\mu) \geq \psi^{2}(S_{j_{0}}).$$
   In this way, after $m$ steps not reaching a contradiction we obtain
   $$S_{j_{0}+m} \geq \psi^{m}(S_{j_{0}}).$$
   As $\psi$ is uniformly expanding in the interval $(\theta^{-1}, 0)$ (actually $\psi'(t)> \psi'(\theta^{-1})=\frac{1}{\left(\theta^{-1}\right)^2}=\theta^2 >1$) we must have $S_{j_{0}+m} > 0$ for a large enough iterate $m$ (see Figure~\ref{fig:recurrence-psi}), a contradiction with $S_{j}<0, j \geq 1$.
   \begin{figure}[H]
  \centering
  \includegraphics[width=12cm]{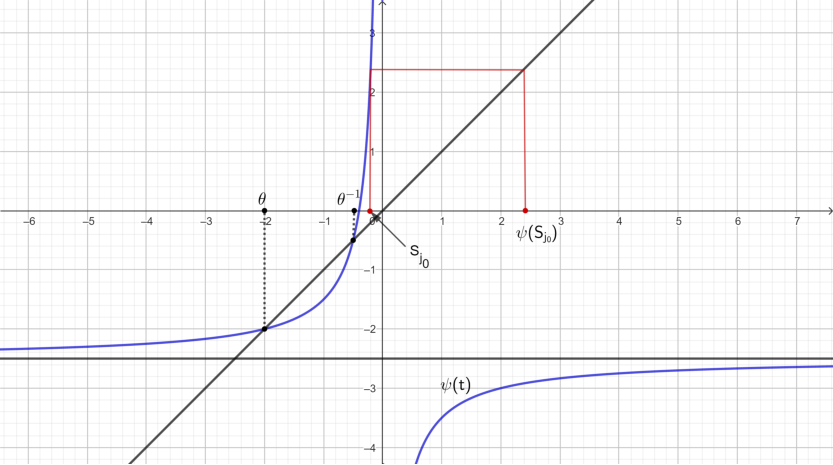}
  \caption{Representation of $\psi$. In this case $S_{j_{0}+m} > 0$ for $m=1$.}\label{fig:recurrence-psi}
\end{figure}
\end{proof}

\section{Density of Laplacian limit points}\label{sec:density laplac lim points}
The main goal of this section is to investigate, for a generalized Shearer sequence $G_{k}(\mathbb{T},C)$
   \[
   \left\{
     \begin{array}{ll}
       G_{1}(\mathbb{T}):=[T_{1}] \\
       G_{k}(\mathbb{T}):=[T_{1},T_{2},..., T_{k-1}, C_{k}], \; k \geq 2,
     \end{array}
   \right.
   \]
where  $\mathbb{T}=(T_{1},T_{2},...) \in \mathcal{L}_{\infty}$ and $C=(C_{2},C_{3},...) \in \mathcal{L}_{\infty}$, the Laplacian limit points $\mu:=\displaystyle\lim_{k \to \infty} \rho_{L}(G_{k})$, which we denoted by $\mathcal{S}$ (see Definition~\ref{def:generalized shearer sequence limits}).

To justify our choices, we start by pointing out that for a
starlike $T_{j} \in \mathbb{T}$, its position, that is dependent on the index $j$, affects $\rho_{L}(G_{k})$ in different ways. In fact, if $j=1$ then $ {\rm deg}(v_1)=1+\omega(T_{j})$ and $\rho_{L}(G_{k}) \geq 2+\omega(T_{j})$. However, if $j>1$ then ${\rm deg}(v_j)=2+\omega(T_{j})$ and $\rho_{L}(G_{k}) \geq 3+\omega(T_{j})$. Thus, in order to stay within an interval $[m, m+1]$ we must control $\omega(T_{j})$ across the entire sequence. For instance for $\mathbb{T}=([1,1],T_{2},...) \in \mathcal{L}_{\infty}$ we can have limit points in $[4,5]$ but we need to require $\omega(T_{j})\leq 1$ for $j \geq 2$, otherwise all the limit points will be in $[5,\infty)$. In other words every $G_k$ must be a quipu!

Another important remark is that if $\mathbb{T}=(T_{1},T_{2},...) \in \mathcal{L}_{\infty}$ is constant or pre-periodic, we can always compute explicitly the limit point, often being an algebraic number. Recall that we already show that for   $\mathbb{T}=([1,1,1],[1]^{\infty},\ldots) \in \mathcal{L}_{\infty}$ and $C=([1,1],[1,1]^{\infty},\ldots) \in \mathcal{L}_{\infty}$ we obtain
$$\displaystyle\lim_{k \to \infty} \rho_{L}(G_{k})=\mu_*$$
where $\mu_*=\frac{5+\sqrt{33}}{2} (= 5.3722+)$ is the larger root of  the polynomial $x^2-5x-2$.

Other  remarkable case is the sequence $\mathbb{T}=([0]^{\infty},\ldots) \in \mathcal{L}_{\infty}$  and $$C=([1,1],[1,2],[1,3],\ldots) \in \mathcal{L}_{\infty}.$$ A close inspection shows that\\
$G_{1}:=[[0]]$;\\
$G_{2}:=[[0], [1,1]]=[1,1,1]$;\\
$G_{3}:=[[0],[0], [1,2]]=[1,2,2]$\\
and so on. Thus $G_{k}:=[[0]^{k-1}, [1,k-1]]=[1,k-1,k-1]$. Obviously $G_{k}$ satisfy the subgraph condition constituting a generalized Shearer sequence. By \cite{OliveTrevAppDiff}, Theorem 4.2, we know that
$$\displaystyle\lim_{k \to \infty} \rho_{L}(G_{k})=2+\epsilon= 4.382975+$$
where $\epsilon$ is the real root of $x^3-4x-4$. It is worth mentioning that the convergence is quite fast, for instance $\rho_{L}([[0], [0], [0], [0], [0], [0], [0], [0], [1, 8]])=4.3829331+$ is already $10^{-4}$ close to the actual value!

\begin{example}\label{ex: a particular family}
Motivated by the above considerations we may take a closer look at the following family of  generalized Shearer sequences:
\begin{equation}\label{eq:family interval}
   \mathcal{F}_{1}:=\{G_{k}(\mathbb{T}, C) \; | \; T_1=[1,1,1], \; T_{j}\in \{[0], [1], [1,1]\}, j\geq 2,\; \text{ and } C \text{ is a shift of }\mathbb{T}\}
\end{equation}
As we observed before $\displaystyle\lim_{k \to \infty} \rho_{L}(G_{k})>5$ for any $G_{k} \in \mathcal{F}_{1}$, thus the question is the distribution of these limit points in the interval $[5, 5+m], \; m \geq 1$.

To give some idea on that distribution we run a numerical experiment sampling 3000 sequences for  $G_{k} \in \mathcal{F}_{1}$ and plot, in the vertical coordinate the value $\rho_{L}(G_{100})$ in Figure~\ref{fig:distrib f1}.
\begin{figure}[H]
  \centering
  \includegraphics[width=5cm]{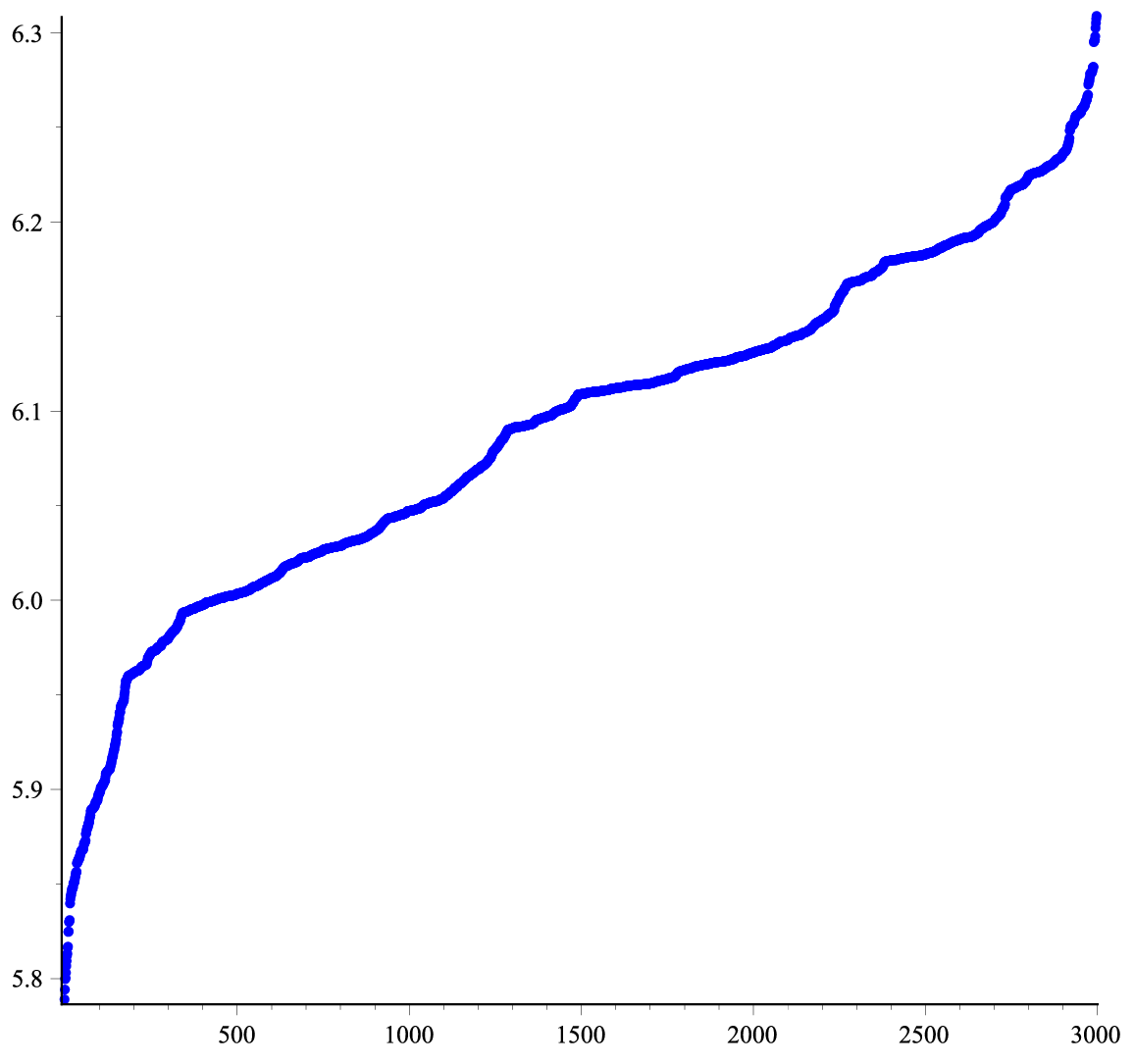}\hspace{0.5cm}
  \includegraphics[width=5cm]{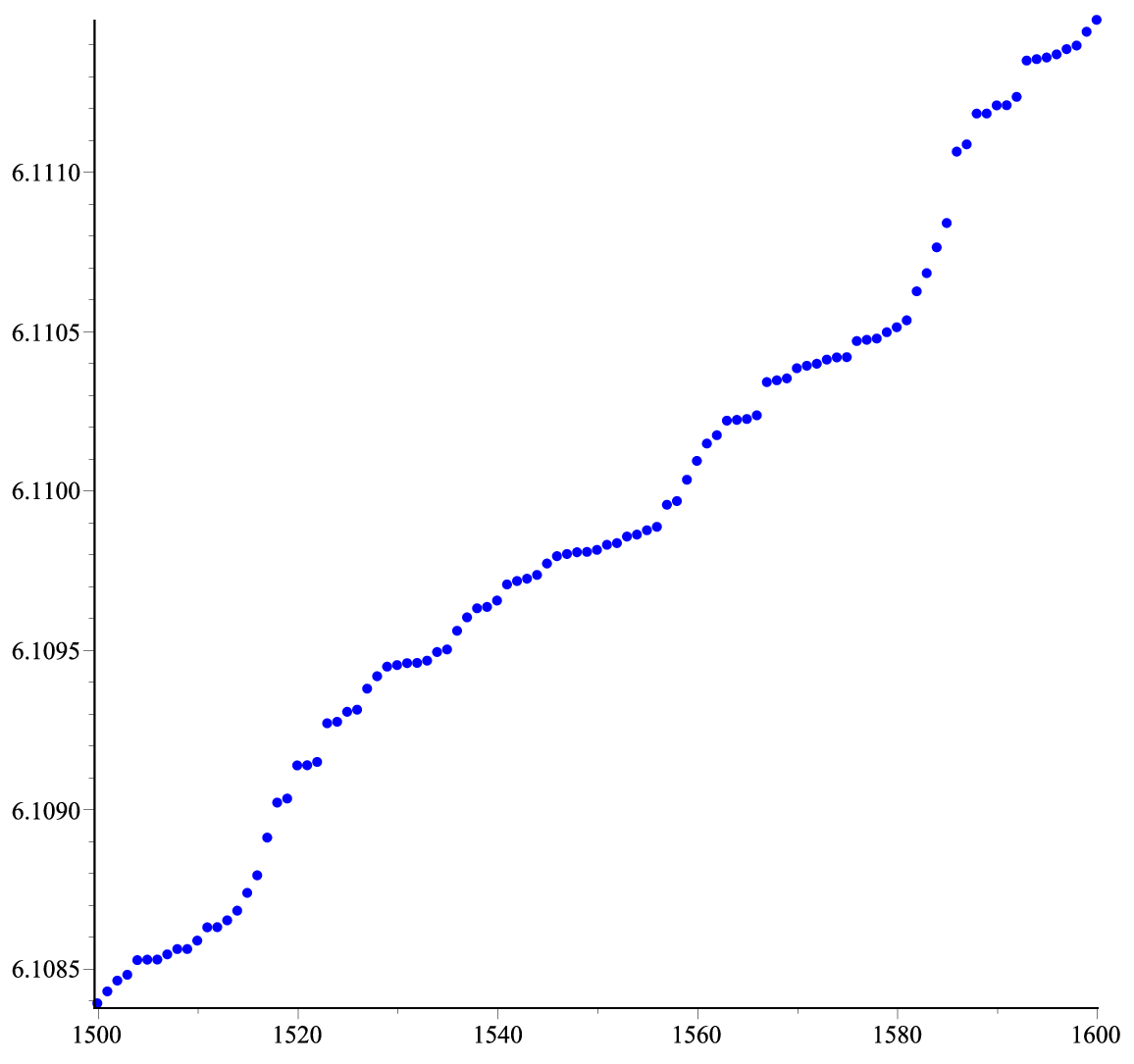}
  \caption{Limit points distribution (in the left) and a zoom of it (in the right).}\label{fig:distrib f1}
\end{figure}
Figure~\ref{fig:distrib f1} suggests an abundance of limit points, but we need a rigorous study to quantify it. In the experiment  $\{[0], [1], [1,1]\}$ were sorted with equal probability, which may affect the variability (see the zoom on the right of Figure~\ref{fig:distrib f1}). The smallest gap is $1.4 \times 10^{-8}$ while the largest is  $0.01+$. We limit in 3000 the number of samples, because larger simulations are computationally prohibitive and pointless to establish density results.
\end{example}

In order to compare different limit points, we start with some general properties regarding the perturbation of a generalized Shearer sequence. In the next result we will use a truncation process, where a sequence $\mathbb{T}=(T_{1},T_{2},T_{3},\ldots) $ is replaced by $\tilde{\mathbb{T}}=(T_{1},T_{2},T_{3},\ldots,T_{k_{0}}, [0]^{\infty})\in \mathcal{L}_{\infty}$ always having a smaller limit point because, after $k_{0}$ there is no drift in the recurrence relation used to compute ${\rm Diagonalize}(\hat{G}_{k}, -\mu_{0})$ for $\displaystyle\lim_{k \to \infty} \rho_{L}(G_{k})=\mu_{0}$. Then, $S_{k_0}<\theta^{-1}$ means that $S_{k_0  +i} =\psi(S_{k_0})<\psi(\theta^{-1})= \theta^{-1}$, for all $ i \geq 1$.
Perhaps more interesting is the fact that this truncation process provides a set of algebraic limit points that is dense in a subset of $\mathcal{S}$ (see Definition~\ref{def:random laplacian shearer limit points}).

\begin{theorem}\label{thm: disturb shearer sequence}
   Let $G_{k}(\mathbb{T}, C)$ be a generalized Shearer sequence where $C=(T_{2},T_{3}, \ldots)$ is a shift of  $\mathbb{T}=(T_{1},T_{2},T_{3},\ldots) $ and $\displaystyle\lim_{k \to \infty} \rho_{L}(G_{k})=\mu_{0} \in \mathcal{S}$. Consider $\varepsilon>0$ and $k_{\varepsilon} \geq 1$ such that $\rho_{L}(G_{k_\epsilon})> \mu_{0} - \varepsilon$.
   \begin{enumerate}
     \item[$(a)$] If $\hat{\mathbb{T}}=(T_{1},T_{2},T_{3},\ldots, T_{j_{0}-1},\hat{T}_{j_{0}}, T_{j_{0}+1}, \ldots)$, defines the sequence $\hat{G}_{k}:=G_{k}(\hat{\mathbb{T}})$ and $\delta(\hat{T}_{j_{0}}, \mu_{0})> \delta(T_{j_{0}}, \mu_{0})$ then $\displaystyle \lim_{k \to \infty} \rho_{L}(\hat{G}_{k}) \geq \mu_{0}$;
     \item[$(b)$] If $\tilde{\mathbb{T}}=(T_{1},T_{2},T_{3},\ldots,T_{k_{\varepsilon}}, [0]^{\infty})$, defines the sequence $\tilde{G}_{k}:=G_{k}(\tilde{\mathbb{T}})$ and $$\displaystyle\lim_{k \to \infty} \rho_{L}(\tilde{G}_{k})=\mu_{1}$$ then $\mu_{1} \in (\mu_{0} - \varepsilon , \mu_{0}]$;
     \item[$(c)$] The set $\mathcal{A}$ of Laplacian limit points of sequences defined by $$\tilde{\mathbb{T}}=(T_{1},T_{2},T_{3},\ldots,T_{k_{0}}, [0]^{\infty})\in \mathcal{L}_{\infty}$$ is algebraic;
     \item[$(d)$] The set $\mathcal{A}$ is dense in the set $\mathcal{B}$ of Laplacian limit points of sequences defined by  $\mathbb{T}=(T_{1},T_{2},\ldots) \in \mathcal{L}_{\infty}$ and $C=(T_{2},T_{3},\ldots) \in \mathcal{L}_{\infty}$.
   \end{enumerate}
\end{theorem}
\begin{proof}
   (a) Since $\hat{G}_{k}=G_{k}$ for $k < j_{0}$, we may assume $k > j_{0}$. This time we will apply ${\rm Diagonalize}(G_{k}, -\mu)$ and ${\rm Diagonalize}(\hat{G}_{k}, \mu)$ with $\mu:=\rho_{L}(G_{k})$, take $v_{j_{0}}$ as root, $f_{G_{k}}(v_{j_{0}})=0$ the final value in $G_{k}$, and $f_{\hat{G}_{k}}(v_{j_{0}})$ the final value in $\hat{G}_{k}$. By Lemma~\ref{lem:loc}, we know that $\rho_{L}(\hat{G}_{k})> \rho_{L}(G_{k})$ if $f_{\hat{G}_{k}}(v_{j_{0}})>0$.

   By direct computation, we obtain
   $$0=f_{G_{k}}(v_{j_{0}})= 2-\mu -\frac{1}{f_{G_{k}}(v_{j_{0} -1})} -\frac{1}{f_{G_{k}}(v_{j_{0}+1})} + \delta(T_{j_{0}}, \mu),$$
   or
   $$2-\mu -\frac{1}{f_{G_{k}}(v_{j_{0} -1})} -\frac{1}{f_{G_{k}}(v_{j_{0}+1})} = -\delta(T_{j_{0}}, \mu).$$
   On the other hand, using the above equation we get
   $$f_{\hat{G}_{k}}(v_{j_{0}})= 2-\mu -\frac{1}{f_{G_{k}}(v_{j_{0} -1})} -\frac{1}{f_{G_{k}}(v_{j_{0}+1})} + \delta(\hat{T}_{j_{0}}, \mu)= \delta(\hat{T}_{j_{0}}, \mu) - \delta(T_{j_{0}}, \mu) >0,$$
   by hypothesis.\\

   (b) The easy part is to prove that $\mu_{1} > \mu_{0} - \varepsilon$. We just need to observe that, by hypothesis, $\rho_{L}(G_{k_{\varepsilon}})> \mu_{0} - \varepsilon$ and for $k > k_{\varepsilon}$ we have $\tilde{G}_{k}=[T_{1},T_{2},T_{3},\ldots,T_{k_{\varepsilon}}, [0]^{k-k_{\varepsilon}}]$, meaning that, $G_{k_{\varepsilon}}$ is a subgraph of  $\tilde{G}_{k}$. Thus, $\displaystyle\mu_{1} = \lim_{k \to \infty} \rho_{L}(\tilde{G}_{k}) \geq \rho_{L}(G_{k_{\varepsilon}})> \mu_{0} - \varepsilon$.

   In order to see that $\mu_{1} \leq \mu_{0}$ we need a more sophisticated argument. Again we apply ${\rm Diagonalize}(G_{k}, -\mu_{0})$ and ${\rm Diagonalize}(\hat{G}_{k}, -\mu_{0})$ taking $v_{k}$ as root in both graphs.
   Let $S_1,\ldots, S_{k-1}, S_{k}$ the output of  ${\rm Diagonalize}(G_{k}, -\mu_{0})$ at the vertices $v_{1}, \ldots, v_{k-1}, v_{k}$ and  $\tilde{S}_{1}, \ldots, \tilde{S}_{k-1}, \tilde{S}_{k}$ the output of  ${\rm Diagonalize}(\tilde{G}_{k}, -\mu_{0})$ at the vertices $v_{1}, \ldots, v_{k-1}, v_{k}$.

From Lemma~\ref{lem:converse gen random} we know that  if $\displaystyle\lim_{k \to \infty} \rho_{L}(G_{k})=\mu_{0}$ and   $\theta_{0}:= \frac{-(\mu_{0} -2) -\sqrt{(\mu_{0} -2)^2 - 4}}{2}$, then
   \[
\left\{
  \begin{array}{ll}
    S_{1}=1  -\mu_{0} +  \delta(T_{1},\mu_{0}) < \theta^{-1}_{0}; \\
    \psi(S_{j-1}) + \delta(T_{j},\mu_{0}) < \theta^{-1}_{0},\; 2 \leq j \leq k.
  \end{array}
\right.
\]
   Assuming $k >  k_{\varepsilon}$ (if $k \leq  k_{\varepsilon}$ then $\rho_{L}(\tilde{G}_{k}) = \rho_{L}(G_{k}))$, we have $\tilde{S}_{j}=S_{j}<\theta^{-1}_{0}<0, \; 1 \leq j \leq k_{\varepsilon}$. For the next index, we obtain
   $$S_{k_{\varepsilon}+1}= \psi(S_{k_{\varepsilon}}) + \delta(T_{k_{\varepsilon}+1},\mu_{0}) <  \theta^{-1}_{0}$$
   and
   $$\tilde{S}_{k_{\varepsilon}+1}= \psi(S_{k_{\varepsilon}}) + \delta([0],\mu_{0})= \psi(S_{k_{\varepsilon}}) <  \theta^{-1}_{0}.$$
   Proceeding in this way we obtain
   $$\tilde{S}_{k_{\varepsilon}+m}= \psi(S_{k_{\varepsilon}+m-1}) + \delta([0],\mu_{0})= \psi^{m-1}(S_{k_{\varepsilon}}) <  \theta^{-1}_{0}.$$
   In particular, as $\psi(S_{k-1}) < \theta^{-1}_{0}$ implies that
   $\tilde{S}_{k}= -1+ \psi(S_{k-1})   < -1+ \theta^{-1}_{0}< 0$.
   Thus $\rho_{L}(\tilde{G}_{k}) < \mu_{0}$ for $k \geq  k_{\varepsilon}$ and $\displaystyle \mu_{1}=\lim_{k \to \infty} \rho_{L}(\tilde{G}_{k}) \leq \mu_{0}$.

   (c) Take $\mu \in \mathcal{A}$, then there exists $\mu_{k}:=\rho_{L}(\tilde{G}_{k})$, defined by $\tilde{\mathbb{T}}=(T_{1},T_{2},T_{3},...,T_{k_{0}}, [0]^{\infty})\in \mathcal{L}_{\infty}$ such that $\displaystyle \mu=\lim_{k \to \infty} \rho_{L}(\tilde{G}_{k})$.

   Consider $S_{1}, \ldots, S_{k-1}, S_{k}$ the output of  ${\rm Diagonalize}(\tilde{G}_{k}, -\mu_{k})$ at the vertices $v_{1}, \ldots, v_{k-1}, v_{k}$. We know that $S_{1}, \ldots, S_{k_{0}-1}, S_{k_{0}}$ are negative rational functions with respect to the variable $\mu_{k}$.  Now, $S_{k_{0}+1} = \psi(S_{k_{0}})+\delta([0],\mu_{k})= \psi(S_{k_{0}})$, $S_{k_{0}+2} = \psi(S_{k_{0}+1})+\delta([0],\mu_{k})= \psi(S_{k_{0}+1})$ until we reach $S_{k} =-1+ \psi(S_{k-1})+\delta([0],\mu_{k})= -1+ \psi(S_{k-1}) =0,$ by construction, or equivalently  $\mu_{k}$ must solve the equation
   $$\psi(S_{k-1})=1.$$

   Defining $a_{1}= S_{k_{0}}$, $a_{2}= S_{k_{0}+1} = \psi(a_{1})$, \ldots, $a_{k-k_{0}}= S_{k_{0}+(k-k_{0})} = \psi(S_{k-1})$, where   $\psi(t)=2-\mu_{k} -\frac{1}{t}, t\neq 0$.  From \cite{OliveTrevAppDiff} we know that $\Delta=(2- \mu_{k})^2-4>0$ thus
$$a_{j}:=\theta_{k}  +\frac{\theta_{k}^{-1} -\theta_{k}}{\beta_{k}\theta_{k}^{2(j -1)}+1},\; 1 \leq j,$$
where $\beta_{k}:=\frac{\theta_{k}^{-1} -\theta_{k}}{a_{1} -\theta_{k}} -1$ and $\theta_{k}=\frac{(2- \mu_{k})- \sqrt{(2- \mu_{k})^2-4}}{2}$.

   Now we are able to rewrite the equation $\psi(S_{k-1})=1$  as $a_{k-k_{0}}=1$ or
   $$\theta_{k}  +\frac{\theta_{k}^{-1} -\theta_{k}}{\beta_{k}\theta_{k}^{2(k-k_{0}-1)}+1} =1$$
   $$ \theta_{k}^{-1} -\theta_{k}  =(1 -\theta_{k}) \left(\beta_{k}\theta_{k}^{2(k-k_{0}-1)}+1\right)$$
   $$ \theta_{k}^{-1} -\theta_{k}  =(1 -\theta_{k}) \left(\left( \frac{\theta_{k}^{-1} -\theta_{k}}{S_{k_{0}} -\theta_{k}} -1 \right)\theta_{k}^{2(k-k_{0}-1)}+1\right)$$
   $$ (\theta_{k}^{-1} -\theta_{k})(S_{k_{0}} -\theta_{k})  =
   (1 -\theta_{k}) \left(\left( \theta_{k}^{-1} - S_{k_{0}}\right)\theta_{k}^{2(k-k_{0}-1)}+  S_{k_{0}} -\theta_{k}\right)$$
   $$ \frac{(\theta_{k}^{-1} -\theta_{k})(S_{k_{0}} -\theta_{k})}{\theta_{k}^{2(k-k_{0}-1)}}  =
   (1 -\theta_{k}) \left(\left( \theta_{k}^{-1} - S_{k_{0}}\right)+  \frac{S_{k_{0}} -\theta_{k}}{\theta_{k}^{2(k-k_{0}-1)}}\right).$$
   Notice that $(\theta_{k}^{-1} -\theta_{k})= \sqrt{(2- \mu_{k})^2-4} \not\to 0$ when $k \to \infty$. Thus, or  $ S_{k_{0}} -\theta_{k}\to 0$ when $k \to \infty$ or taking the limit in the above equation we obtain $ \theta_{k}^{-1} - S_{k_{0}} \to 0$ when $k \to \infty$.
   Thus $\displaystyle \mu=\lim_{k \to \infty} \rho_{L}(\tilde{G}_{k})$ must solve one of the equations
   $$\theta^{-1} =  S_{k_{0}}\text{ or }\theta =  S_{k_{0}}.$$
   All in all, as $S_{k_{0}}= \frac{P(\mu)}{Q(\mu)}$ we conclude that $\mu$ must solve
   $$\frac{(2- \mu)\pm \sqrt{(2- \mu)^2-4}}{2} = \frac{P(\mu)}{Q(\mu)}$$
   $$\pm \sqrt{(2- \mu)^2-4} = \frac{2 P(\mu)}{Q(\mu)} - (2- \mu)$$
   $$(2- \mu)^2-4 = \left(\frac{2 P(\mu)}{Q(\mu)} - (2- \mu)\right)^{2}$$
   $$(2- \mu)^2-4 = 4\left(\frac{P(\mu)}{Q(\mu)}\right)^{2} - 4(2- \mu)\left(\frac{P(\mu)}{Q(\mu)}\right) + (2- \mu)^{2}$$
   $$-4 = 4\left(\frac{P(\mu)}{Q(\mu)}\right)^{2} - 4(2- \mu)\left(\frac{P(\mu)}{Q(\mu)}\right) $$
   $$\left(\frac{P(\mu)}{Q(\mu)}\right)^{2} - (2- \mu)\left(\frac{P(\mu)}{Q(\mu)}\right)+1 =0 $$
   $$P^2(\mu) - (2- \mu)\left(P(\mu)Q(\mu)\right)+ Q^2(\mu) =0.$$
   After a simplification we conclude that $\mu$ is a real root of the polynomial at the numerator, meaning that $\mu$ is an algebraic number.

   (d) Consider $\mu_{0} \in \mathcal{B}$  and $\varepsilon > 0$ arbitrarily small. By (b) we can find $k_{\varepsilon}$  such that $\tilde{\mathbb{T}}_{\varepsilon}=(T_{1},T_{2},T_{3},...,T_{k_{\varepsilon}}, [0]^{\infty})$ defines a sequence $\tilde{G}_{k}:=G_{k}(\tilde{\mathbb{T}_{\varepsilon}})$ such that $\displaystyle\lim_{k \to \infty} \rho_{L}(\tilde{G}_{k})=\mu_{1} \in (\mu_{0} - \varepsilon , \mu_{0}]$. As $\mu_{1} \in \mathcal{A}$, this proves the density.
\end{proof}

The next general result is on the set of possible limit points \emph{dominated} by $\mu$.
\begin{definition}\label{def: dominated limit points}
   Let $\mu \geq 4.38+$ be a fixed real number. A generalized Shearer sequence  $G_{k}(\mathbb{T}, C)$ is dominated by $\mu$ if  $\displaystyle \lim_{k \to \infty} \rho_{L}(G_{k}) \leq \mu$. We denote by $\Gamma(\mu)$ the set of all generalized Shearer sequences  $G_{k}(\mathbb{T}, C)$ which are dominated by $\mu$.
\end{definition}
We notice that if $\Gamma(\mu)$ is known, this means we completely understand the approximation of $\mu$ by Laplacian limit points originated by linear trees.
In other words, if
$$\sup_{G_{k}\in \Gamma(\mu)}  \lim_{k \to \infty} \rho_{L}(G_{k}) < \mu$$
then either $\mu$ is not a Laplacian limit point or the approximation must be taken through a wider class such as general trees or general graphs. On the other hand if
$$\sup_{G_{k}\in \Gamma(\mu)}  \lim_{k \to \infty} \rho_{L}(G_{k}) = \mu$$
then $\mu$ is certainly a Laplacian limit point, but we do not know whether $\mu \in \mathcal{S}$ that is, whether the supremum is attained by a sequence of linear trees.
\begin{lemma}\label{lem: dominated limit points} Let $\mu \geq 4.38+$ be a fixed real number and let $G_{k}(\mathbb{T}, C)$ be a generalized Shearer sequence where $C=(T_{2},T_{3}, \ldots)$ is a shift of  $\mathbb{T}=(T_{1},T_{2},T_{3},\ldots) $.
   \begin{enumerate}
     \item[$(a)$] If $G_{k}(\mathbb{T}, C) \in \Gamma(\mu)$ then
     $$\omega(T_{j}) \leq  \delta(T_j,\mu) < \mu,$$
     for all $j\geq 1$. The same is true for $C$, that is, $\omega(C_k) \leq  \delta(C_k,\mu) < \mu$;

     \item[$(b)$] Consider $G_{k}(\mathbb{T}, C) \in \Gamma(\mu)$ and a new sequence $$\tilde{\mathbb{T}}=(T_{1},T_{2},T_{3},\ldots, T_{j_{0}-1},\tilde{T}_{j_{0}}, T_{j_{0}+1}, \ldots)$$ defining $\tilde{G}_{k}:=G_{k}(\tilde{\mathbb{T}})$. If $\delta(\tilde{T}_{j_{0}}, \mu)< \delta(T_{j_{0}}, \mu)$ then $\tilde{G}_{k} \in \Gamma(\mu)$;

     \item[$(c)$] Consider $G_{k}(\mathbb{T}, C) \in \Gamma(\mu)$ and starlike trees $\hat{T}_{j_{0}}, \hat{T}_{j_{0}+1}$ such that  $\delta(\hat{T}_{j_{0}}, \mu) \leq \delta(T_{j_{0}}, \mu)$ and
     $\displaystyle\theta^{-1} - \psi(\hat{S}_{j_{0}})> \delta(\hat{T}_{j_{0}+1}, \mu) \geq \delta(T_{j_{0}+1}, \mu).$
There exist some $m \in \mathbb{N}$ and starlike trees $\hat{T}_{j_{0}+2}, \ldots,\hat{T}_{j_{0}+m}$ such that, defining $$\hat{\mathbb{T}}=(T_{1},T_{2},T_{3},\ldots, T_{j_{0}-1},\hat{T}_{j_{0}}, \hat{T}_{j_{0}+1}, \ldots,\hat{T}_{j_{0}+m}, T_{j_{0}+m+1}, \ldots),$$ we obtain a generalized Shearer sequence $\hat{G}_{k}:=G_{k}(\hat{\mathbb{T}},C)$ with  $\hat{G}_{k} \in \Gamma(\mu)$.
   \end{enumerate}
\end{lemma}
\begin{proof}
  (a) Recall that $\Pi(G_k,\mu) = \left(S_1, S_2,\ldots, S_k \right)$  where
  \[\left\{
    \begin{array}{ll}
      S_{1}=1-\mu + \delta(T_1,\mu),\\
      S_{j}=\psi(S_{j-1})+ \delta(T_j,\mu),\; 2 \leq j \leq k-1,\\
      S_{k}=-1+ \psi(S_{k-1})+ \delta(C_k,\mu).
    \end{array}
  \right.\]
    As $G_{k}(\mathbb{T}, C) \in \Gamma(\mu)$, we have  from Lemma~\ref{lem:converse gen random} that
    $$1-\mu + \delta(T_1,\mu) < \theta^{-1},$$
     but
    $$ \theta^{-1} - (1-\mu )= \frac{2-\mu + \sqrt{(2-\mu)^2 -4}}{2}  - (1-\mu)=$$
$$ = \frac{2-\mu  - 2(1-\mu) + \sqrt{(2-\mu)^2 -4}}{2} = \frac{\mu  + \sqrt{(2-\mu)^2 -4}}{2}$$ thus
    $$ \omega(T_1) \leq   \delta(T_1,\mu) < \frac{\mu + \sqrt{(2-\mu)^2 -4}}{2}.$$
    By construction, $S_{j}> S_{1}> 1-\mu$ for all $j$ and $S_{j}<  \theta^{-1} \Rightarrow \psi(S_{j-1})<  \theta^{-1}$ thus
    $$ 1-\mu + \delta(T_j,\mu)< S_{j}=\psi(S_{j-1})+ \delta(T_j,\mu) < \theta^{-1}$$
     $$\omega(T_j) \leq  \delta(T_j,\mu) < \theta^{-1} - (1-\mu ) =\frac{\mu + \sqrt{(2-\mu)^2 -4}}{2}.$$
Notice that
$$\frac{\mu + \sqrt{(2-\mu)^2 -4}}{2}= \frac{\mu }{2} \left( 1 + \sqrt{\left(\frac{2}{\mu}-1\right)^2 -\left(\frac{2}{\mu}\right)^2} \right)=$$
$$=\frac{\mu }{2} \left( 1 +
\sqrt{\left(\frac{2}{\mu}-1 + \frac{2}{\mu} \right)\left( \frac{2}{\mu}-1 - \frac{2}{\mu} \right)} \right)= \frac{\mu }{2} \left( 1 +
\sqrt{\left(1 - \frac{4}{\mu} \right)} \right) \leq \frac{\mu }{2} 2 = \mu,$$
because $\frac{4}{\mu}<1$.

Finally, $S_{k}=-1+ \psi(S_{k-1})+ \delta(C_k,\mu) <0$  and $1-\mu < S_{k-1}<  \theta^{-1} \Rightarrow 1-\mu < \psi(S_{k-1})<  \theta^{-1}$ thus $-\mu< -1+\psi(S_{k-1}) < -1+ \theta^{-1}$. So,
$$-1+ \psi(S_{k-1})+ \delta(C_k,\mu) <0$$
$$\delta(C_k,\mu) <-(-1+ \psi(S_{k-1}))< \mu$$
$$\omega(C_k)\leq  \delta(C_k,\mu) <-(-1+ \psi(S_{k-1}))< \mu.$$

(b) Consider $\Pi(G_k,\mu) = \left(S_1, S_2,\ldots, S_k \right)$  where (by Lemma~\ref{lem:converse gen random})
  \[\left\{
    \begin{array}{ll}
      S_{1}=1-\mu + \delta(T_1,\mu) < \theta^{-1},\\
      S_{j}=\psi(S_{j-1})+ \delta(T_j,\mu)<\theta^{-1},\; 2 \leq j \leq k-1,\\
      S_{k}=-1+ \psi(S_{k-1})+ \delta(C_k,\mu) < 0.
    \end{array}
  \right.\]
Now  we compute $\Pi:=(\tilde{G}_k,\mu) \mapsto \left(\tilde{S}_1, \tilde{S}_2,\ldots, \tilde{S}_k \right)$. We know that $\tilde{S}_j=S_j < \theta^{-1}<0 $ for $1\leq j_0-1$ and
$$\tilde{S}_{j_{0}}=\psi(S_{j_{0}-1})+ \delta(\tilde{T}_{j_{0}},\mu) -\delta(T_{j_{0}},\mu) + \delta(T_{j_{0}},\mu)= $$
$$=S_{j_{0}} +  \delta(\tilde{T}_{j_{0}},\mu) -\delta(T_{j_{0}},\mu)< S_{j_{0}} +0 < \theta^{-1}<0.$$
Since $\tilde{S}_{j_{0}}< S_{j_{0}}$, $\tilde{T}_{j_{0}+m}=T_{j_{0}+m},~ \forall m \geq 1$ and because $\psi$ is order preserving, we have $\tilde{S}_{j_{0}+m}< S_{j_{0}+m}< \theta^{-1}<0$,
$$\tilde{S}_{j_{0}+m}< S_{j_{0}+m}< \theta^{-1}$$
$$\psi(\tilde{S}_{j_{0}+m})< \psi(S_{j_{0}+m})< \psi(\theta^{-1})$$
$$\psi(\tilde{S}_{j_{0}+m})+0< \psi(S_{j_{0}+m})+0< \theta^{-1}$$
$$\tilde{S}_{j_{0}+(m+1)}< S_{j_{0}+(m+1)}< \theta^{-1}.$$
In particular $\tilde{S}_{j_{0}+1}< S_{j_{0}+1}< \theta^{-1}<0$. This concludes the proof because $\sign(\Pi):=(\tilde{G}_k,\mu) =\left(-, -,\ldots, - \right)$  thus $\rho_{L}(\tilde{G}_k) < \mu$  and $\displaystyle\lim_{k \to \infty }\rho_{L}(\tilde{G}_k) \leq  \mu$ or $\tilde{G}_k \in \Gamma(\mu)$.

(c) We have $\Pi(G_k,\mu) = \left(S_1, S_2,\ldots, S_k \right)$  where
  \[\left\{
    \begin{array}{ll}
      S_{1}=1-\mu + \delta(T_1,\mu)<\theta^{-1},\\
      S_{j}=\psi(S_{j-1})+ \delta(T_j,\mu)<\theta^{-1},\; 2 \leq j \leq k-1,\\
      S_{k}=-1+ \psi(S_{k-1})+ \delta(C_k,\mu)<0.
    \end{array}
  \right.\]
At the same time we have $\Pi(\hat{G}_k,\mu)= \left(\hat{S}_1, \hat{S}_2,\ldots, \hat{S}_k \right)$. For $1 \leq j_0 -1$ we have $\hat{S}_j =S_j <\theta^{-1}$ and for $\hat{S}_{j_{0}}$ we have
$$S_{j_{0}} = \psi(S_{j_{0}-1})+ \delta(T_{j_{0}},\mu)<\theta^{-1}$$
and by hypothesis
\begin{equation}\label{eq:delta balance}
  \delta(\hat{T}_{j_{0}}, \mu)< \delta(T_{j_{0}}, \mu) < \theta^{-1} - \psi(S_{j_{0}-1})
\end{equation}

On the other hand
$$\hat{S}_{j_{0}} = \psi(S_{j_{0}-1})+ \delta(\hat{T}_{j_{0}},\mu)< S_{j_{0}}$$
$$\psi(\hat{S}_{j_{0}})< \psi(S_{j_{0}})$$
$$\theta^{-1} - \psi(\hat{S}_{j_{0}})> \theta^{-1}- \psi(S_{j_{0}}).$$
Using this inequality in Equation~\eqref{eq:delta balance} we obtain
$$\delta(\hat{T}_{j_{0}}, \mu)< \delta(T_{j_{0}}, \mu) < \theta^{-1} - \psi(S_{j_{0}-1}) < \theta^{-1} - \psi(\hat{S}_{j_{0}})$$
$$\hat{S}_{j_{0}} = \psi(\hat{S}_{j_{0}}) + \delta(\hat{T}_{j_{0}}, \mu) < \theta^{-1}.$$

The next is $\hat{S}_{j_{0}+1} = \psi(\hat{S}_{j_{0}}) + \delta(\hat{T}_{j_{0}+1}, \mu)$. In order to have $\hat{S}_{j_{0}+1} < \theta^{-1}$ we need $\delta(\hat{T}_{j_{0}+1}, \mu) < \theta^{-1} - \psi(\hat{S}_{j_{0}})$.

Notice that
$$\theta^{-1} - \psi(\hat{S}_{j_{0}})> \theta^{-1}- \psi(S_{j_{0}})> \delta(T_{j_{0}+1}, \mu).$$

It is always possible  $\theta^{-1} - \psi(\hat{S}_{j_{0}}) > \delta(\hat{T}_{j_{0}+1}, \mu)$ (or equivalently $\hat{S}_{j_{0}+1} < \theta^{-1}$)  provided that
$$\theta^{-1} - \psi(\hat{S}_{j_{0}})> \delta(\hat{T}_{j_{0}+1}, \mu) \geq \delta(T_{j_{0}+1}, \mu),$$
in the worst case one may take $\hat{T}_{j_{0}+1}=T_{j_{0}+1}$.

Assuming our hypothesis, we know that $\hat{S}_{j_{0}+1} < \theta^{-1}$ we may take, in the worst case, $\hat{T}_{j_{0}+2}=[0]$, $\hat{T}_{j_{0}+3}=[0]$, ...,  $\hat{T}_{j_{0}+ m+1}=[0]$. This means that $\hat{S}_{j_{0}+2}= \psi(\hat{S}_{j_{0}+1}) < \theta^{-1}$, $\hat{S}_{j_{0}+3}= \psi^2 (\hat{S}_{j_{0}+1}) < \theta^{-1}$, ..., $\hat{S}_{j_{0}+ m+1}= \psi^m (\hat{S}_{j_{0}+1}) < \theta^{-1}$. As $m$ increases $\hat{S}_{j_{0}+ m+1}= \psi^m (\hat{S}_{j_{0}+1}) \to \theta < \theta^{-1}$. Thus, unless the original sequence is formed only by $[0]$'s  after  $T_{j_{0}+2}$ will exist such a $m$ in the way that $\hat{S}_{j_{0}+ m+1} < S_{j_{0}+ m+1} < \theta^{-1}$ for all the next indices. Finally we conclude that $\sign(\Pi)(\hat{G}_k,\mu)= \left(-, -,\ldots, - \right)$ for all $k$ thus $\hat{G}_k \in \Gamma(\mu)$.\\
\end{proof}
\begin{remark}
  In item (c) of Lemma~\ref{lem: dominated limit points} we can take $\hat{T}_{j_{0}}=T_{j_{0}}$ and the proof would still be valid. At the same time, taking $\delta(\hat{T}_{j_{0}}, \mu) < \delta(T_{j_{0}}, \mu),$ we will be able to choose $\hat{T}_{j_{0}+1}$ in a wider range.  Moreover, one can take $\hat{T}_{j_{0}+1}$ in such a way that we maximize the quantity  $\delta(\hat{T}_{j_{0}+1}, \mu)$ in the range $$[\delta(T_{j_{0}+1}, \mu), \; \theta^{-1} - \psi(\hat{S}_{j_{0}})).$$
Finally, instead of taking the easier choice $\hat{T}_{j_{0}+2}=[0]$, $\hat{T}_{j_{0}+3}=[0]$, \ldots,  $\hat{T}_{j_{0}+ m+1}=[0]$, one can simply take
$$\delta(\hat{T}_{j_{0}+m+1}, \mu) < \delta(T_{j_{0}+m+1}, \mu), \; m \geq 1,$$
with $m$ such that $\hat{S}_{j_{0}+ m+1} < S_{j_{0}+ m+1} $  may be reached much later.
\end{remark}

\section{Variational technique}\label{sec:vartec}
The goal of this section is to investigate, for a generalized Shearer sequence $G_{k}(\mathbb{T},C)$
   \[
   \left\{
     \begin{array}{ll}
       G_{1}(\mathbb{T}):=[T_{1}] \\
       G_{k}(\mathbb{T}):=[T_{1},T_{2},..., T_{k-1}, C_{k}], \; k \geq 2,
     \end{array}
   \right.
   \]
where  $\mathbb{T}=(T_{1},T_{2},...) \in \mathcal{L}_{\infty}$ and $C=(C_{2},C_{3},...) \in \mathcal{L}_{\infty}$, when the Laplacian limit points $\mu_{0}:=\displaystyle\lim_{k \to \infty} \rho_{L}(G_{k})$ is close to a given number $\mu\geq 4.38+$.

In order to do that, we assume $G_{k} \in \Gamma(\mu)$ and consider a variation $\mu_{\varepsilon}:=\mu-\varepsilon$. We ask whether
 $\mu_{0}> \mu_{\varepsilon}$. We already know that $\mu_{0}> \mu_{\varepsilon}$ if, and only if, there exists some $k \in \mathbb{N}$ such that the vector
$$\sign(\Pi)(G_k,\mu_{\varepsilon})$$
has some positive entry, where $\Pi(G_k,\mu_{\varepsilon}) = \left(S_1(\varepsilon), S_2(\varepsilon),\ldots, S_k(\varepsilon) \right)$  is given by
  \[\left\{
    \begin{array}{ll}
      S_{1}(\varepsilon)=1-\mu_{\varepsilon} + \delta(T_1,\mu_{\varepsilon}),\\
      S_{j}(\varepsilon)=\psi_{\varepsilon}(S_{j-1}(\varepsilon))+ \delta(T_j,\mu_{\varepsilon}),\; 2 \leq j \leq k-1,\\
      S_{k}(\varepsilon)=-1+ \psi_{\varepsilon}(S_{k-1}(\varepsilon))+ \delta(C_k,\mu_{\varepsilon}).
    \end{array}
  \right.\]

We need to take a closer look at some terms at the above formula:\\
\begin{itemize}
  \item  The function $\psi_{\varepsilon}$ is given by $\psi_{\varepsilon}(t)= 2-\mu_{\varepsilon} - \frac{1}{t} = 2-\mu - \frac{1}{t} + \varepsilon = \psi(t) + \varepsilon$. Of course, $\psi_{0}(t)= \psi(t)$.
  \item The processing of each path in $T_i$ satisfies the formula
\begin{equation}\label{eq:path-epsilon}
    \left\{ \begin{array}{ll}
               b_1(\varepsilon) = 1-\mu_{\varepsilon}\\
               b_{j+1}(\varepsilon) = \psi_{\varepsilon}(b_j(\varepsilon)), ~~j>1.
             \end{array}
    \right.
   \end{equation}
By a reasoning analogous to the case $\varepsilon=0$, one can prove that for a small $\varepsilon$, the iterates $b_j(\varepsilon)$ obeying Equation~\eqref{eq:path-epsilon} has the following properties. Let $\theta_{\varepsilon}$ and $\theta'_{\varepsilon}$ be fixed points of $\psi_{\varepsilon}(t)=t$. They are given by
\begin{equation}\label{eq:theta-epsilon}
\theta_{\varepsilon}:= \frac{-(\mu_{\varepsilon} -2) -\sqrt{(\mu_{\varepsilon} -2)^2 -4}}{2} \mbox { and } \theta'_{\varepsilon}:= \theta_{\varepsilon}^{-1}>\theta_{\varepsilon}.
\end{equation}
They are attracting and repelling points, respectively, meaning that for $b_1(\varepsilon) \leq  \theta'_{\varepsilon}$, we have $b_j(\varepsilon)$ converging to $\theta_{\varepsilon}$.
Since $\mu> 4.38+$ we have $\mu_{\varepsilon} > 4$, so that
$$1-\mu_{\varepsilon} < b_1(\varepsilon) \leq b_j(\varepsilon) < \theta_{\varepsilon}<0 $$
and, in particular, that all $b_j(\varepsilon) < 0$.
    \item   The function $\varepsilon \to  b_j(\varepsilon)$ is differentiable and $\frac{d\, b_j(\varepsilon)}{d\,\varepsilon }\geq 1$. In particular, $\varepsilon \to  b_j(\varepsilon)$ is strictly increasing. Define $b'_j(\varepsilon):= \frac{d\, b_j(\varepsilon)}{d\,\varepsilon}$. Then, the sequence $(b'_{j}(\varepsilon))$ satisfy the recurrence relation
        \begin{equation}\label{eq: deriv prim bj}
        \left\{
          \begin{array}{ll}
            b'_{1}(\varepsilon)=1 \\
            b'_{j}(\varepsilon) = 1+\frac{1}{\left(b_{j-1}(\varepsilon)\right)^2} b'_{j-1}(\varepsilon),\; j \geq 2.
          \end{array}
        \right.
        \end{equation}
Indeed, $b'_{1}(\varepsilon)= \frac{d\, (1-\mu_{\varepsilon})}{d\,\varepsilon }=1>0$. For $j >1$ we have $b'_{j}(\varepsilon)=\frac{d\, \psi_{\varepsilon}(b_{j-1}(\varepsilon))}{d\,\varepsilon }=1+\frac{1}{b_{j-1}(\varepsilon)^2} b'_{j-1}(\varepsilon)$.  By induction we get $b'_{j}(\varepsilon) \geq 1$, for $j \geq 1$.\\
Additionally, we can also set a upper bound to $b'_{j}(\varepsilon)$:\\
$$b_j(\varepsilon) < \theta_{\varepsilon}<0  \Rightarrow 0<\frac{1}{-b_j(\varepsilon)} < \frac{1}{-\theta_{\varepsilon}} \Rightarrow \frac{1}{(b_j(\varepsilon))^{2}} < \theta_{\varepsilon}^{-2}.$$
Thus
$$b'_{j}(\varepsilon) = 1+\frac{1}{\left(b_{j-1}(\varepsilon)\right)^2} b'_{j-1}(\varepsilon) \leq 1+ \theta_{\varepsilon}^{-2} b'_{j-1}(\varepsilon),$$
which means that $b'_{j}(\varepsilon) \leq 1+ \theta_{\varepsilon}^{-2} + \ldots + \theta_{\varepsilon}^{-2j} \leq \frac{1}{1- \theta_{\varepsilon}^{-2}}<\infty$.
    \item   The function $\varepsilon \to  b_j(\varepsilon)$ is twice differentiable and $\frac{d^{2}\, b_j(\varepsilon)}{d\,\varepsilon^{2} }\geq 0$. In particular, if $j>1$ then $\varepsilon \to  b_j(\varepsilon)$ is strictly concave, because $\frac{d\, b_j(\varepsilon)}{d\,\varepsilon }$ is not constant.
Define $b''_j(\varepsilon):= \frac{d^{2}\, b_j(\varepsilon)}{d\,\varepsilon^{2}}$. Then, the sequence $(b''_{j}(\varepsilon))$ satisfies the recurrence relation
        \begin{equation}\label{eq: deriv seg bj}
        \left\{
          \begin{array}{ll}
            b''_{1}(\varepsilon)=0 \\
            b''_{j}(\varepsilon)=- \frac{2}{b_{j-1}(\varepsilon)^3} \left(b'_{j}(\varepsilon)\right)^{2} + \frac{1}{b_{j-1}(\varepsilon)^2}\;  b''_{j-1}(\varepsilon)  ,\; j \geq 2.
          \end{array}
        \right.
        \end{equation}
Indeed, we already computed $b'_{1}(\varepsilon)=1$ thus $b''_{1}(\varepsilon)= \frac{d\,}{d\,\varepsilon }(1)=0$. \\
        For $j >1$ we have $$b''_{j}(\varepsilon)=\frac{d\,}{d\,\varepsilon }\left(1+\frac{1}{b_{j-1}(\varepsilon)^2} b'_{j-1}(\varepsilon)\right)= $$
        $$=- \frac{2}{b_{j-1}(\varepsilon)^3} \left(b'_{j}(\varepsilon)\right)^{2} + \frac{1}{b_{j-1}(\varepsilon)^2}\;  b''_{j-1}(\varepsilon).$$
        By induction we get $b''_{j}(\varepsilon) \geq 0$, for $j \geq 1$.

  \item  Let $T_{j}$ be one starlike tree attached to $G_{k}$. Recall that $\omega(T_{j})$ denote the number of paths of $T_{j}=[q_{1}^{j},\ldots,q_{\omega(T_{j})}]$ and, for $i=1,\ldots,\omega(T_{j})$, by $b_{q_{i}^{j}}$ the last value output by the algorithm ${\rm Diagonalize}(G_{k},-\mu_{\varepsilon})$ at the $i$-th path of $T_{j}$. We denote the net result of the starlike $T_{j}$ by
\begin{equation}\label{eq:defdelta-epsilon}
\delta(T_{j}, \mu_{\varepsilon}):=\left\{
\begin{array}{ll}
       \sum_{i=1}^{\omega(T_{j})} \left(1 -\frac{1}{ b_{q_{i}^{j}}(\varepsilon)}\right),  \mbox{ for non-empty } T_{j} \\
       0 \mbox{ if } T_{j} \mbox{ is empty. }
\end{array} \right.
\end{equation}
From the previous items we can define the differential drift, $ \delta'(T_{j}, \mu_{\varepsilon}):= \frac{d\,}{d\,\varepsilon } \delta(T_{j}, \mu_{\varepsilon})$ which satisfies the  formula for $T_{j} \neq [0]$
\begin{equation}\label{eq:deriv drift}
\delta'(T_{j}, \mu_{\varepsilon})= \sum_{i=1}^{\omega(T_{j})} \frac{1}{ \left(b_{q_{i}^{j}}(\varepsilon)\right)^{2}} b'_{q_{i}^{j}}(\varepsilon) \geq 0.
\end{equation}
We can also find an upper bound for $\delta'(T_{j}, \mu_{\varepsilon})$. Recall that $\frac{1}{(b_j(\varepsilon))^{2}} < \theta_{\varepsilon}^{-2}$ and $b'_j(\varepsilon) <\frac{1}{1- \theta_{\varepsilon}^{-2}}$, thus
$$\delta'(T_{j}, \mu_{\varepsilon})= \sum_{i=1}^{\omega(T_{j})} \frac{1}{ \left(b_{q_{i}^{j}}(\varepsilon)\right)^{2}} b'_{q_{i}^{j}}(\varepsilon) \leq  \sum_{i=1}^{\omega(T_{j})} \theta_{\varepsilon}^{-2} \frac{1}{1- \theta_{\varepsilon}^{-2}} = \frac{\omega(T_{j})}{\theta_{\varepsilon}^{2} - 1} < \infty$$
because $\omega(T_{j})< \infty$ (see Lemma~\ref{lem:boundness seq starlikes}).
\end{itemize}

\begin{theorem}\label{thm: criteria approx}
     Let $G_{k}:=G_{k}(\mathbb{T},C)$ be a generalized Shearer sequence  such that $(G_{k}) \in \Gamma(\mu)$. Then $\mu:=\displaystyle\lim_{k \to \infty} \rho_{L}(G_{k})$ if and only if the sequence $(\varepsilon_{j})$ converges to zero, where $\varepsilon_{j}$ is the smallest positive solution of  $S_j(\varepsilon)=0$, for $\Pi(G_k,\mu_{\varepsilon}) = \left(S_1(\varepsilon), S_2(\varepsilon),\ldots, S_k(\varepsilon) \right)$.
\end{theorem}
\begin{proof}
First, we will investigate the sign of $S_{1}(\varepsilon)=1-\mu_{\varepsilon} + \delta(T_1,\mu_{\varepsilon})$.  In this case,
$$\sign(S_{1}(\varepsilon))=+ \Leftrightarrow 1-\mu_{\varepsilon} + \delta(T_1,\mu_{\varepsilon}) > 0 \Leftrightarrow \varepsilon > \mu -1 - \delta(T_1,\mu_{\varepsilon}).$$
To study the above inequality we define the function $g_{1}:[0,\infty) \to \mathbb{R}$ given by
$$g_{1}(\varepsilon)=S_{1}(\varepsilon))= 1-\mu_{\varepsilon} + \delta(T_1,\mu_{\varepsilon}).$$
It is easy to see that $g_{1}(0)= 1-\mu + \delta(T_1,\mu) < \theta^{-1} <0$ $(G_{k}) \in \Gamma(\mu)$. Also, $g_{1}(\varepsilon)$ is differentiable in some neighborhood of  $\varepsilon=0$.
Indeed, for $T_{1}=[q_{1}^{1},\ldots, q_{\omega(T_1)}^{1}]$ we get
$$\frac{d\, g_{1}(\varepsilon)}{d\,\varepsilon }= 1 + \delta'(T_{1}, \mu_{\varepsilon}) \geq 1.$$

Since $\frac{d\, g_{1}(\varepsilon)}{d\,\varepsilon }\geq 1$ and $g_{1}(0)<0$ we obtain an unique root $\varepsilon_{1}$ such that $g_{1}(\varepsilon) <0$ for $\varepsilon< \varepsilon_{1}$ and $g_{1}(\varepsilon) >0$ for $\varepsilon> \varepsilon_{1}$. In this way, we conclude that $\sign(S_{1}(\varepsilon))=+$ for $\varepsilon > \varepsilon_{1}$, in other words, $\sign(\Pi)(G_k,\mu_{\varepsilon}) =\left(+,?,\ldots, ? \right)$, meaning that, $\rho_{L}(G_{k})> \mu_{\varepsilon}$.

In order to get a better approximation of $\mu$ we need to look for the possibility of having $\sign(S_{2}(\varepsilon))=+$ for $\varepsilon< \varepsilon_{1}$. Analogously to the previous case, we define $$g_{2}(\varepsilon)=S_{2}(\varepsilon)= \psi_{\varepsilon}(S_{1}(\varepsilon))+ \delta(T_2,\mu_{\varepsilon}),\; 0\geq \varepsilon< \varepsilon_{1}.$$
It is easy to see that $g_{2}(0)= S_{2}(0))=S_{2} < \theta^{-1} <0$ because  $(G_{k}) \in \Gamma(\mu)$.
For $T_{2}=[q_{1}^{2},\ldots, q_{\omega(T_2)}^{2}]$ we get
$$\frac{d\, g_{2}(\varepsilon)}{d\,\varepsilon }= 1 +\frac{1}{ \left(g_{1}(\varepsilon)\right)^2} \frac{d\, g_{1}(\varepsilon)}{d\,\varepsilon } +
\delta'(T_{2}, \mu_{\varepsilon})  \geq 1.$$
Again, as  $\frac{d\, g_{2}(\varepsilon)}{d\,\varepsilon }\geq 1$ and $g_{2}(0)<0$ we obtain a unique root $\varepsilon_{2}$ such that $g_{2}(\varepsilon) <0$ for $\varepsilon < \varepsilon_{2}$ and $g_{2}(\varepsilon) >0$ for $\varepsilon> \varepsilon_{2}$.\\
We claim that $\varepsilon_{2} < \varepsilon_{1}$. To see that, we analyze
$$\lim_{\varepsilon \to \varepsilon_{1}^{-}} g_{2}(\varepsilon) = \lim_{\varepsilon \to \varepsilon_{1}^{-}} 2-\mu - \frac{1}{g_{1}(\varepsilon)} + \varepsilon + \delta(T_2,\mu_{\varepsilon}) = +\infty $$
because $g_{1}(\varepsilon)<0$ for $\varepsilon < \varepsilon_{1}$ and $\displaystyle\lim_{\varepsilon \to \varepsilon_{1}^{-}} g_{1}(\varepsilon)=0$. Thus, $g_{2}(\varepsilon)>0$ in some left neighborhood of $\varepsilon_{1}$, that is, the root $\varepsilon_{2} < \varepsilon_{1}$.\\
In this way, we conclude that $\sign(S_{2}(\varepsilon))=+$ for $\varepsilon_{2}< \varepsilon < \varepsilon_{1}$, in other words, $\sign(\Pi)(G_k,\mu_{\varepsilon}) =\left(-,+,?,\ldots, ? \right)$, meaning that, $\rho_{L}(G_{k})> \mu_{\varepsilon}$.

If we continue this procedure for $2 \leq j \leq k-1$ we obtain a sequence $0< \varepsilon_{j}< \ldots \varepsilon_{2} < \varepsilon_{1}$ such that, for $\varepsilon_{j}< \varepsilon < \varepsilon_{j-1}$ we have $\sign(\Pi)(G_k,\mu_{\varepsilon}) =\left(-,\ldots, -, +,?,\ldots, ? \right)$ (obviously, by construction, $\sign(\Pi)(G_k,\mu_{\varepsilon}) =\left(-,\ldots, -, -,?,\ldots, ? \right)$ for $\varepsilon< \varepsilon_{j}$), meaning that, $\rho_{L}(G_{k})> \mu_{\varepsilon}$.

The last step is to look for the possibility of having $\sign(S_{k}(\varepsilon))=+$ for $\varepsilon< \varepsilon_{k-1}$. Analogously to the previous case, we define $$g_{k}(\varepsilon)=S_{k}(\varepsilon)= -1+ \psi_{\varepsilon}(S_{k-1}(\varepsilon))+ \delta(C_k,\mu_{\varepsilon}),\; 0\geq \varepsilon< \varepsilon_{k-1}.$$
It is easy to see that $g_{k}(0)= S_{k}(0))=S_{k} < 0$ because  $(G_{k}) \in \Gamma(\mu)$.
For $C_{k}=[p_{1}^{k},\ldots, p_{\omega(C_k)}^{k}]$ we get
$$\frac{d\, g_{k}(\varepsilon)}{d\,\varepsilon }= 1 +\frac{1}{ \left(g_{k-1}(\varepsilon)\right)^2} \frac{d\, g_{k-1}(\varepsilon)}{d\,\varepsilon } + \delta'(C_{k}, \mu_{\varepsilon})  \geq 1.$$
Again, as  $\frac{d\, g_{k}(\varepsilon)}{d\,\varepsilon }\geq 1$ and $g_{k}(0)<0$ we obtain a unique root $\varepsilon_{k}$ such that $g_{k}(\varepsilon) <0$ for $\varepsilon < \varepsilon_{k}$ and $g_{k}(\varepsilon) >0$ for $\varepsilon> \varepsilon_{k}$.\\
We claim that $\varepsilon_{k} < \varepsilon_{k-1}$. To see that, we analyze
$$\lim_{\varepsilon \to \varepsilon_{k-1}^{-}} g_{k}(\varepsilon) = \lim_{\varepsilon \to \varepsilon_{k-1}^{-}} 2-\mu - \frac{1}{g_{k-1}(\varepsilon)} + \varepsilon + \delta(C_k,\mu_{\varepsilon}) = +\infty $$
because $g_{k-1}(\varepsilon)<0$ for $\varepsilon < \varepsilon_{k-1}$ and $\displaystyle\lim_{\varepsilon \to \varepsilon_{k-1}^{-}} g_{k-1}(\varepsilon)=0$. Thus, $g_{k}(\varepsilon)>0$ in some left neighborhood of $\varepsilon_{1}$, that is, the root $\varepsilon_{k} < \varepsilon_{k-1}$.\\
In this way, we conclude that $\sign(S_{k}(\varepsilon))=+$ for $\varepsilon_{k}< \varepsilon < \varepsilon_{k-1}$, in other words, $\sign(\Pi)(G_k,\mu_{\varepsilon}) =\left(-,\ldots, -, + \right)$, meaning that, $\rho_{L}(G_{k})> \mu_{\varepsilon}$.

The conclusion is that, $\displaystyle\lim_{k \to \infty} \rho_{L}(G_{k})=\mu$ if, and only if, the sequence $(\varepsilon_{j})$ converges to zero. Otherwise, if $\varepsilon_{j} \geq \varepsilon_{\infty}>0$ we get $\rho_{L}(G_{k}) < \mu - \varepsilon_{\infty}$ thus $\mu_{0}= \displaystyle\lim_{k \to \infty} \rho_{L}(G_{k})< \mu$.
\end{proof}
\begin{remark}
  Notice that  $\varepsilon_k$  is the solution of
  $$S_{k}(\varepsilon_{k})= -1+ \psi_{\varepsilon}(S_{k-1}(\varepsilon_{k}))+ \delta(C_k,\mu_{\varepsilon_{k}})=0$$
  meaning that
  $$-\varepsilon_{k}:=-1+ \psi(S_{k-1}(\varepsilon_{k}))+ \delta(C_k,\mu_{\varepsilon_{k}})$$
  If  $\varepsilon_{j} \to 0$ then $\mu_{\varepsilon_{k}}=\mu -\varepsilon_{k} \to \mu$ and $S_{k}= -1+ \psi(S_{k-1})+ \delta(C_k,\mu) = [-1+ \psi(S_{k-1})+ \delta(C_k,\mu) -(-1+ \psi(S_{k-1}(\varepsilon_{k}))+ \delta(C_k,\mu_{\varepsilon_{k}}))] -  \varepsilon_{k}  \to 0$ .
\end{remark}

For practical purposes, it is feasible, from a theoretical point of view, but computationally very hard, to obtain the numbers $(\varepsilon_{j})$, because each function $g_{j}(\varepsilon)$ is obtained recursively from the previous one. We notice, however, that we can compute another sequence $(\alpha_{j})$ such that $\varepsilon_{j} < \alpha_{j}$. Thus, if $(\alpha_{j})$ converges to zero then $(\varepsilon_{j})$ converges to zero, providing a sufficient condition to ensure that $\displaystyle\lim_{k \to \infty} \rho_{L}(G_{k})=\mu$. The numbers $(\alpha_{j})$ will be the root of the linear approximation $h_{j}(\varepsilon):=g_{j}(0) + \varepsilon \frac{d\, g_{j}}{d\,\varepsilon }(0)$ of $g_{j}(\varepsilon)$ in $\varepsilon=0$, see Figure~\ref{fig:alpha-epsilon}.
\begin{figure}[H]
  \centering
  \includegraphics[width=12cm]{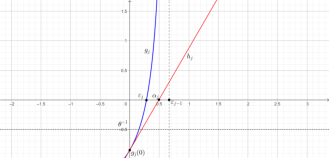}
  \caption{Representation of $\alpha_{j}$.}\label{fig:alpha-epsilon}
\end{figure}
\begin{corollary}\label{cor:suff condition to converge}
   Under the hypothesis of Theorem~\ref{thm: criteria approx}, consider the sequence $(\varepsilon_{j})$. Let $(\alpha_{j})$ be the sequence of roots of the functions
   $$h_{j}(\varepsilon):=g_{j}(0) + \varepsilon \frac{d\, g_{j}}{d\,\varepsilon }(0).$$
   Then, $\varepsilon_{j} < \alpha_{j}$ and, if $(\alpha_{j})$ converges to zero, then $\displaystyle\lim_{k \to \infty} \rho_{L}(G_{k})=\mu$. Moreover,
$\displaystyle\alpha_{1}:= \frac{-S_{1}}{1+ \delta'(T_{1}, \mu_{\varepsilon})(0)}$ and, for $2 \leq j \leq k$,
\begin{equation}\label{eq:recur-alpha second}
  \alpha_{j}= \frac{-S_{j}}{\beta_{j} -\frac{1}{ S_{j-1}} \frac{1}{\alpha_{j-1}}},
\end{equation}
where $\beta_{j}:=1+ \delta'(T_{j}, \mu_{\varepsilon})(0)$, and for $2 \leq j \leq k-1$, $\beta_{k}:=1+ \delta'(C_{k}, \mu_{\varepsilon})(0)$.
\end{corollary}
\begin{proof}
   By definition, $\alpha_{j}$ is the solution of $h_{j}(\alpha_{j})=0$, which is given by
   $$\alpha_{j}= \frac{-g_{j}(0)}{\frac{d\, g_{j}}{d\,\varepsilon }(0)}.$$
   We recall that $g_{j}(0)=S_{j}$,
  $$\frac{d\, b_j}{d\,\varepsilon }(0)=
    1+\frac{1}{b_{j-1}^2} \frac{d\, b_{j-1}}{d\,\varepsilon }(0)$$
   and
   \[\frac{d\, g_{j}}{d\,\varepsilon }(\varepsilon)=
   \left\{
     \begin{array}{ll}
       1 + \delta'(T_{1}, \mu_{\varepsilon})(0), & j=1, \; T_{1}=[q_{1}^{1},\ldots, q_{\omega(T_1)}^{1}] \\
        1 +\frac{1}{ \left(g_{j-1}(\varepsilon)\right)^2} \frac{d\, g_{j-1}(\varepsilon)}{d\,\varepsilon } +
\delta'(T_{j}, \mu_{\varepsilon})(0), & 2 \leq j \leq k-1, \; T_{j}=[q_{1}^{j},\ldots, q_{\omega(T_j)}^{j}]\\
       1 +\frac{1}{ \left(g_{k-1}(\varepsilon)\right)^2} \frac{d\, g_{k-1}(\varepsilon)}{d\,\varepsilon } +
\delta'(C_{k}, \mu_{\varepsilon})(0), & j=k, \; C_{k}=[p_{1}^{k},\ldots, p_{\omega(C_k)}^{k}]
     \end{array}
   \right.
   \]
so
\begin{equation}\label{eq: deriv g_j in zero}
  \frac{d\, g_{j}}{d\,\varepsilon }(0)= 1 +\frac{1}{ \left(S_{j-1}\right)^2} \frac{d\, g_{j-1}}{d\,\varepsilon }(0) +
\delta'(T_{j}, \mu_{\varepsilon})(0).
\end{equation}
Substituting the above formulas, for $2 \leq j \leq k-1$, we obtain
$$\alpha_{j}= \frac{-S_{j}}{1 +\frac{1}{ \left(S_{j-1}\right)^2} \frac{d\, g_{j-1}}{d\,\varepsilon }(0) +\delta'(T_{j}, \mu_{\varepsilon})(0)}.$$
Notice that  $\displaystyle\alpha_{j-1}= \frac{-g_{j-1}(0)}{\frac{d\, g_{j-1}}{d\,\varepsilon }(0)}=\frac{-S_{j-1}}{\frac{d\, g_{j-1}}{d\,\varepsilon }(0)}$ so
\begin{equation}\label{eq:recur-alpha first}
  \alpha_{j}= \frac{-S_{j}}{\beta_{j} -\frac{1}{ S_{j-1}} \frac{1}{\alpha_{j-1}}},
\end{equation}
where $\beta_{j}:=1 + \delta'(T_{j}, \mu_{\varepsilon})(0)$.
The above formula provides a way to obtain $\alpha_{j}$ from $\alpha_{j-1}$ without compositions, only the computation of $S_{j}$, which is easy to perform, and the knowledge of the derivatives $b'_j(0)$, which is a very standard recursion problem (see Equation~\eqref{eq: deriv prim bj}, for $\varepsilon=0$), independent from this particular construction.\\
In order to conclude our proof we just need to show that $\varepsilon_{j} < \alpha_{j}$. We claim that it is true if $g_{j}(\varepsilon)$ is a concave function ($\frac{d^{2}\, }{d\,\varepsilon^{2} }(g_{j}) \geq 0$). Indeed, $h_{j}(0)=g_{j}(0)$ and for $\varepsilon>0$ we have
$$\frac{d\, }{d\,\varepsilon }(g_{j}(\varepsilon)-h_{j}(\varepsilon))=\frac{d\, g_{j}}{d\,\varepsilon } - \frac{d\, g_{j}}{d\,\varepsilon }(0) \geq 0.$$
Thus  $g_{j}(\alpha_{j})-h_{j}(\alpha_{j})= g_{j}(\alpha_{j}) \geq g_{j}(0)-h_{j}(0)= 0$, so that $g_{j}(\alpha_{j}) \geq 0$. Since $g_{j}$ is increasing we get  $\varepsilon_{j} < \alpha_{j}$. \\
In this way, we only need to show that $\frac{d^{2}\, }{d\,\varepsilon^{2} }(g_{j}) \geq 0$:
$$ \frac{d^{2}\, g_{j}(\varepsilon)}{d\,\varepsilon^{2} } =\frac{d\, }{d\,\varepsilon }\left(1 +\frac{1}{ \left(g_{j-1}(\varepsilon)\right)^2} \frac{d\, g_{j-1}(\varepsilon)}{d\,\varepsilon } +\sum_{i=1}^{\omega(T_j)} \frac{1}{ \left(b_{q_{i}^{j}}(\varepsilon)\right)^2} \frac{d\, b_{q_{i}^{j}}(\varepsilon)}{d\,\varepsilon }\right)= $$
$$=\frac{-2}{ \left(g_{j-1}(\varepsilon)\right)^3} \left(\frac{d\, g_{j-1}(\varepsilon)}{d\,\varepsilon }\right)^{2} +\frac{1}{ \left(g_{j-1}(\varepsilon)\right)^2} \frac{d^{2}\, g_{j-1}(\varepsilon)}{d^{2}\,\varepsilon } +$$
$$\sum_{i=1}^{\omega(T_j)}
\left( \frac{-2}{ \left(b_{q_{i}^{j}}(\varepsilon)\right)^3} \left(\frac{d\, b_{q_{i}^{j}}(\varepsilon)}{d\,\varepsilon }\right)^{2}
+ \frac{1}{ \left(b_{q_{i}^{j}}(\varepsilon)\right)^2} \frac{d^{2}\, b_{q_{i}^{j}}(\varepsilon)}{d\,\varepsilon^{2} }\right)  \geq 0,$$
because $b_{j}<0$ for all $j$, $\frac{d^{2}\, }{d\,\varepsilon^{2} }(g_{1}) = 0$ and we already proved that $b''_{j}(\varepsilon) \geq 0$.
\end{proof}

\begin{example} For $\mu=\frac{5+\sqrt{33}}{2}$ and the generalized Shearer sequence $\mathbb{T}=([1,1,1],[1], \ldots)$ and $C=([1,1], \ldots)$ we know that $\displaystyle\lim_{k \to \infty} \rho_{L}(G_{k})=\mu$. We can check, from the formula~\eqref{eq:recur-alpha second} that\\
$\alpha_{1}= 0.5930703308$\\
$\alpha_{10}= 0.0003726377$\\
$\ldots$\\
$\alpha_{100}= 1.33485599 \times 10^{-33}$\\
$\alpha_{150}= 5.84453701 \times 10^{-50}$\\
$\alpha_{190}= 4.78412668 \times 10^{-63}$. Thus, $|\rho_{L}(G_{190})-\mu |\leq 10^{-63}$.\\
On the other hand, using  for $\mu=5.4$ the same sequence (we can do that since $G_{k} \in \Gamma(5.4)$) we get $\alpha_{1}= 0.62182468$, $\alpha_{10}= 0.670219903$, $\alpha_{50}= 0.807268557543450193961813121639$, $\alpha_{100}= 0.807268557543452357547180905158$, \\ $\alpha_{150}= 0.80726855754345235754718090515772416691821457170280270632003$ and so on, as expected, since $\displaystyle\lim_{k \to \infty} \rho_{L}(G_{k})=\frac{5+\sqrt{33}}{2}=5.37+ <5.4$.\\
Finally, we consider a sequence selected from a genetic algorithm to fit to $\mu=5.4$: $$\mathbb{T}=([0], [1, 1], [1], [7], [5], [6], [7], [7], [2], [3], [5], [6], [2], [5], [4], [4], [6],$$ $$ [6], [6], [0], [6], [1, 1], [0], [6], [0], [3], [4], [4], [7], [1, 1], ...)$$ and $C$ a shift of $\mathbb{T}$. In this case $G_{k} \in \Gamma(5.4)$ and $\alpha_{29}= 0.0001005914$, suggesting that $\rho_{L}(G_{29})$is  at least $10^{-4}$ close to $\mu=5.4$. Actually, a direct computation shows that $\rho_{L}(G_{29})=5.399999999963451$, which is $3.65 \times 10^{-11}$ close to $\mu=5.4$. That is natural since the convergence of $\alpha_j \to 0$ is only a sufficient condition and $\varepsilon_j$ could be much smaller than $\alpha_j$.
\end{example}

An alternative way to prove that $\alpha_j \to 0$ and so $\displaystyle\lim_{k \to \infty} \rho_{L}(G_{k})=\mu$, without a direct computation, is the following:
\begin{theorem} \label{thm: alternative to alpha goes to zero}
   If
\begin{equation}\label{eq: solution derivative infinite imply alpha goes to zero}
   \sum_{m=1}^{j}\frac{1 + \delta'(T_{m}, \mu_{\varepsilon})(0)}{\left(\prod_{n=m}^{j-1}S_{n}\right)^2}
\end{equation}
is unbounded, then $\displaystyle\lim_{k \to \infty} \rho_{L}(G_{k})=\mu$.
\end{theorem}
\begin{proof}
   Let $\alpha_j$ be the sequence given by  Corollary~\ref{cor:suff condition to converge}. By definition, $\alpha_{j}$ is given by
   $$\alpha_{j}= \frac{-S_{j}}{\frac{d\, g_{j}}{d\,\varepsilon }(0)}.$$
   From Equation~\eqref{eq: deriv g_j in zero} we get
$$\frac{d\, g_{j}}{d\,\varepsilon }(0)= 1 +\frac{1}{ \left(S_{j-1}\right)^2} \frac{d\, g_{j-1}}{d\,\varepsilon }(0) +
\delta'(T_{j}, \mu_{\varepsilon})(0),$$
and $\frac{d\, g_{1}}{d\,\varepsilon }(0)= 1 +
\delta'(T_{1}, \mu_{\varepsilon})(0)$.\\
Let us introduce the following auxiliary sequences\\
\begin{itemize}
  \item $A_{j}:= 1 + \delta'(T_{j}, \mu_{\varepsilon})(0)$;
  \item $B_{j}:= \frac{1}{ \left(S_{j}\right)^2}$;
  \item $X_{j}:= \frac{d\, g_{j}}{d\,\varepsilon }(0)$.
\end{itemize}
Then, we obtain the standard difference equation problem
\begin{equation}\label{eq: derivative infinite imply alpha goes to zero}
   \left\{
     \begin{array}{ll}
       X_{j}= A_{j} + B_{j-1} X_{j-1}\\
       X_{1}= A_{1}.
     \end{array}
   \right.
\end{equation}
This equation has an explicit solution:\\
$$X_{2}= A_{2} + B_{1} X_{1}=  A_{2} + B_{1} A_{1},$$
$$X_{3}= A_{3} + B_{2} (A_{2} + B_{1} A_{1})= A_{3} + A_{2} B_{2} + A_{1} B_{1} B_{2},$$
$$X_{4}= A_{4} + A_{3} B_{3} + A_{2} B_{2} B_{3}+  A_{1} B_{1} B_{2}B_{3},$$
and so on, obtaining
$$
  \frac{d\, g_{j}}{d\,\varepsilon }(0)= X_{j}= \sum_{m=1}^{j} A_{m} \prod_{n=m}^{j-1} B_{n} = \sum_{m=1}^{j} \left(1 + \delta'(T_{m}, \mu_{\varepsilon})(0)\right) \prod_{n=m}^{j-1} \frac{1}{ \left(S_{n}\right)^2}.
$$
Notice that $S_{1} < S_{j} < \theta^{-1}$ so $-S_{1} > -S_{j} > -\theta^{-1}$. Thus, if $\frac{d\, g_{j}}{d\,\varepsilon }(0)$ is unbounded when $j \to \infty$,  then $\alpha_j \to 0$. By Corollary~\ref{cor:suff condition to converge} we obtain $\displaystyle\lim_{k \to \infty} \rho_{L}(G_{k})=\mu$.
\end{proof}

\begin{remark}
   The formula
$$
  \frac{d\, g_{j}}{d\,\varepsilon }(0)=  \sum_{m=1}^{j}\frac{1 + \delta'(T_{m}, \mu_{\varepsilon})(0)}{\left(\prod_{n=m}^{j-1}S_{n}\right)^2},
$$
has an analogous for adjacency limit points, and the fact that it is  unbounded proves Shearer's result in \cite{shearer1989distribution}. Although the approach is a little different, the basis for Shearer's argument is to show that the product $\prod_{n=m}^{j-1} \left(S_{n}\right)$ has absolute values smaller than 1. This is somehow hidden in the proof of the main theorem (actually he proves that $R_{j-1}R_{j}<1, R_{j-2}R_{j+1}<1,...$ thus $R_{j-2}R_{j-1}R_{j}R_{j+1}<1$, etc,  where the $R_{j}'s$  are the values obtained in the adjacency case). In our Laplacian case, if this was true, we would obtain $\frac{d\, g_{j}}{d\,\varepsilon }(0) \to \infty$ and $\displaystyle\lim_{k \to \infty} \rho_{L}(G_{k})=\mu$  according to  Theorem~\ref{thm: alternative to alpha goes to zero}. Since it is not true  that Equation~\eqref{eq: solution derivative infinite imply alpha goes to zero} is unbounded for any $\mu$, this explains why the classical Laplacian Shearer fails to prove convergence in general.
\end{remark}

\section{Final considerations and future work}\label{sec:final}

In this section we discuss possible consequences of the previous results and additional investigations one could carry on in the future, towards the proof of the conjecture that all the points in the interval $[4.38+, \infty)$ are Laplacian limit points.\\

We mention here an ongoing research work \cite{oliveira2024limit}, where we study limit points for the spectral radius of $A_\alpha$-matrix of a graph $G$, where $A_{\alpha}(G):= \alpha D(G)+ (1-\alpha)A(G)$, for $0\leq \alpha\leq 1$.  It is easy to see that $A_{0}(G)=A(G)$, the adjacency matrix of $G$ and $A_{1}(G)=D(G)$ the degree matrix of $G$. Also, $A_{1/2}(G)=1/2(D(G)+A(G))=1/2 Q(G)$ where $Q$ is the signless Laplacian matrix of $G$. Our main result  is that for any $\alpha \in [0, 1/2)$ there exists a positive number $\tau_2(\alpha)>2$ such that any value $\lambda> \tau_2(\alpha)$ is an $A_{\alpha}$-limit point. Notice that this generalization of  the  Shearer's work  for the adjacency matrix seems to stretch full power (which is $\alpha< \frac{1}{2}$ ). Beyond that, the technique used does not work. As for  $\alpha=\frac{1}{2}$  we have the signless Laplacian matrix, these two manuscripts show how difficult it is to study Laplacian limit points .

\noindent {\bf Question 1}\\

Can one use Theorem~\ref{thm: criteria approx}, or Corollary~\ref{cor:suff condition to converge}, or Theorem~\ref{thm: alternative to alpha goes to zero} to prove that  $$\displaystyle \sup_{G_{k}\in \Gamma(\mu)}  \lim_{k \to \infty} \rho_{L}(G_{k})= \mu $$ at least for some class of points $\mu \geq 4.38$?\\  The answer of this question is definitive for approximation by linear trees (which include caterpillars), using increasing sequences of  spectral radii.\\

\noindent {\bf Question 2}\\

Are there are any points where $\displaystyle \sup_{G_{k}\in \Gamma(\mu)}  \lim_{k \to \infty} \rho_{L}(G_{k}) < \mu?$ \\
If yes, then $\mu$ is not a Laplacian limit point or at least it can not be limit of  increasing sequences of Laplacian spectral radii of linear trees, because $\Gamma(\mu)$ contains all linear trees with spectral radius smaller than $\mu$.\\

\noindent {\bf Question 3}\\

In Question 2, can one find (from the previous lemmas and theorems) a maximum, where the supremum is attained, despite the fact that $\displaystyle \sup_{G_{k}\in \Gamma(\mu)}  \lim_{k \to \infty} \rho_{L}(G_{k}) < \mu$ ?\\
 If yes, this would be a special number in some way!\\

\noindent {\bf Question 4}\\

Is there some particular subinterval of $[4.38+, \infty)$ regarding properties of the product $S_{j}S_{j+1}<1$, provided that $S_{j}>-1$ and $S_{j+1}<-1$?\\ If the answer is yes, then one  may be able to use Theorem~\ref{thm: alternative to alpha goes to zero} to prove that $\displaystyle \sup_{G_{k}\in \Gamma(\mu)}  \lim_{k \to \infty} \rho_{L}(G_{k})= \mu.$\\

\noindent {\bf Question 5}\\

A natural future problem would be to investigate the application of  Theorem~\ref{thm: alternative to alpha goes to zero}. Can one prove that $X_{k}> m_{k} \to \infty$?\\ In order to do that, one needs to control weather $S_{j}S_{j+1}<1$, at least for $j \geq k$ because (for $k \geq 2$) $$X_{4}= A_{4} + A_{3} B_{3} + A_{2} B_{2} B_{3}+  A_{1} B_{1} B_{2}B_{3},$$ we start with $A_{2} B_{2} B_{3}$ where $A_{2}= 1 + \delta'(T_{2}, \mu_{\varepsilon})(0)>1$ and $B_{2} B_{3}= \frac{1}{S_{2}^{2} S_{3}^{2}} >1 $, provided that $S_{2} S_{3} <1$. This does not depend on $S_{1}$, same as for the next additives, except for $ A_{1} B_{1} B_{2}B_{3}$. Recall that the goal is to show that $X_{k} \to \infty$, in this case one may have $X_{4}> 3$ if $B_{3}>1$ because $A_{4}>1$, $A_{2} B_{2} B_{3}>1$, but we do not have control of $A_{1} B_{1} B_{2}B_{3}$ since $S_{1}$ could be negative and as large as $1-\mu$. In the worst case $X_{4}> 2$.\

\noindent {\bf Question 6}\\

Despite the success of the Shearer's method for adjacency limit points, one may ask if part of our difficulties when we translate to the Laplacian limit points, comes from the choice of an iterative process that builds sequences of increasing spectral radii, as we did. One possible line of research is to consider Lapalacian limit points arbitrarily constructed, but still using treatable trees like caterpillars, linear trees and others. If one succeeds, then one may expect a proof of the conjecture or to conclude that one should look at a more general class of trees or general graphs.

\section*{Acknowledgments} This work is partially supported by MATH-AMSUD under project GSA, brazilian team financed by CAPES under project 88881.694479/2022-01.  It was executed while F. Belardo was visiting UFRGS, under the financial of CAPES - PRINT, project 88887.716953/2022-00. The research of Francesco Belardo and Vilmar Trevisan was supported by a grant from the group GNSAGA of INdAM (Italy). Elismar Oliveira and Vilmar Trevisan acknowledge the support of CNPq grants 408180/2023-4 and 310827/2020-5.

\end{document}